\documentclass[11pt]{amsart}

 \usepackage{amsmath,amssymb}
 \usepackage{comment}
\usepackage[margin=3.cm]{geometry}

\usepackage[toc,page,title,titletoc,header]{appendix}

\usepackage{colordvi}
\usepackage[all]{xy}
\usepackage{colortbl}

\newtheorem{innercustomthm}{Theorem}[section]

\newcommand{\poubelle}[1]{}

\usepackage[all]{xy}
\usepackage{colortbl}

\usepackage{amsmath,amssymb,graphics,epsfig,color}
\usepackage{enumerate,wasysym,stmaryrd}

\def\R{\mathbb{R}}

\def\{{\left\lbrace}
\def\}{\right\rbrace}
\DeclareMathOperator{\e}{e}

\newtheorem{theorem}{Theorem}[section]

\newtheorem{corollary}[theorem]{Corollary}

\newtheorem{lemma}[theorem]{Lemma}
\newtheorem{prop}[theorem]{Proposition}

\newenvironment{customthm}[1]
  {\renewcommand\theinnercustomthm{#1}\innercustomthm}
  {\endinnercustomthm}

\renewcommand{\v}{\mathrm{v}}
\newcommand{\w}{\mathrm{w}}

\newcommand{\T}{\mathbb{T}}

\newcommand{\ext}{\mathrm{ext}}

\newcommand{\hess}{\textnormal{Hess}}

\newcommand{\vv}{\mathrm{v}(\mu_\varphi)}

\newcommand{\Scal}{\mathrm{Scal}}
\newcommand{\Ric}{\mathrm{Ric}}
\newcommand{\ric}{\mathrm{ric}}
\newcommand{\PSH}{{\rm PSH}}

\newcommand{\Ent}{{\rm Ent}}
\newcommand{\Osc}{{\rm Osc}}
\newcommand{\capa}{{\rm Cap}}
\newcommand{\vol}{{\rm Vol}}
\newcommand{\Tr}{\Lambda_0}

\newcommand{\tor}{{\mathfrak t}}

\usepackage[utf8]{inputenc}

\title{Weighted cscK metrics (I): a priori estimates}

\author{Eleonora Di Nezza, Simon Jubert and Abdellah Lahdili}

\begin{document}

\maketitle

\begin{abstract}
\noindent Let $X$ be a compact K\"ahler manifold. In this paper we study the existence of constant weighted scalar curvature K\"ahler (weighted cscK) metrics on $X$. More precisely, we establish a priori $C^{k}$-estimates ($k\geq 0$) for the K\"ahler potential associated with these metrics, thereby extending a result due to Chen and Cheng in the classical cscK setting.
\end{abstract}

\tableofcontents

\section{Introduction}{\label{s:intro}}

A central theme in K\"ahler geometry is the search for canonical K\"ahler metrics. The concept of \emph{constant weighted scalar curvature K\"ahler metrics} (weighted cscK for short), introduced by the third author in \cite{Lah19}, provides a unification of various related notions of canonical K\"ahler metrics that satisfy special curvature conditions. 

We present here the first of two papers, focused on investigating the existence of weighted cscK metrics in a given cohomology class. In this paper, we derive \emph{a priori} $C^{k}$-estimates, $k\geq 0$, for the K\"ahler potential associated to a metric with constant weighted scalar curvature. This work extends the result of Chen and Cheng \cite{CC21a} concerning cscK metrics. In the subsequent second paper, we will demonstrate the existence of the weighted cscK metric provided that the weighted Mabuchi energy is relatively coercive, and we will explore several geometric applications. This will be a generalization of Chen-Cheng's continuity path for cscK-metrics. Some works have already been carried on in some special cases: \cite{AJL,Jub23} deals with the semisimple fibration case, \cite{HL} the $g$-solitons case, \cite{He} the extremal metrics case, and \cite{NN24} treats the case of $\sigma$-extremal metrics.

\subsection{Weighted cscK problem}
We fix a compact K\"ahler manifold $(X, \omega_0)$ of complex dimension $n$.  Let $\mathbb{T}$ be an $r$-dimensional compact real torus in the reduced automorphism group $\mathrm{Aut}_{\mathrm{red}}(X)$ with Lie algebra $\mathfrak{t}$. Suppose $\omega_0$ is a $\mathbb{T}$-invariant K\"ahler form and consider the set of smooth $\mathbb{T}$-invariant K\"ahler potentials $\mathcal{K}(X,\omega_0)^{\mathbb{T}}$ relative to $\omega_0$. For $\varphi \in \mathcal{K}(X,\omega_0)^{\mathbb{T}}$ we denote by $\omega_{\varphi}=\omega_0+dd^c\varphi$  the corresponding K\"ahler metric. Here $d=\partial +\overline{\partial}$ and $d^c= i (\overline{\partial}-\partial )$ so that $dd^c= 2i \partial \overline{\partial}$ as in \cite[page 29]{Gau}. This choice of $d$ and $d^c$ is consistent with the fact that we choose the scalar curvature to be defined as two times the trace; note that this normalization has an impact also in the expression of the Ricci form, i.e. $\Ric(\omega_\varphi)=\Ric(\omega_0)-\frac{1}{2} dd^c \log \frac{\omega_\varphi^n}{\omega_0^n}$.

It is well-known that the $\mathbb{T}$-action on $X$ is $\omega_{\varphi}$-Hamiltonian (see e.g. \cite[Section 2.5]{Gau}) and  that $P_{\varphi}:= \mu_{\varphi}(X)$ is a convex polytope in $\mathfrak{t}^{*}$ \cite{Ati, GS, Lah19}. Here $\mu_\varphi: X \rightarrow \mathfrak{t}^{*}$ is the moment map associated to $\omega_\varphi$ and $\mathfrak{t}^{*}$ stands for the dual vector space for the Lie algebra $\mathfrak{t}$ of $\mathbb{T}$. We normalize $\mu_{\varphi}$ by


\begin{equation}{\label{normalizing-moment-map}}
\mu_{\varphi}=\mu_{0} + d^c\varphi,
\end{equation}

\noindent in such a way that $P=P_{\varphi}$ is $\varphi$-independent, see \cite[Lemma 1]{Lah19}. From now on, we always suppose that the moment map satisfies \eqref{normalizing-moment-map}. Let us stress that the identity in \eqref{normalizing-moment-map} means that for all $\xi \in \mathfrak{t}$, 
\begin{equation*}
\langle \mu_{\varphi}, \xi \rangle = \langle \mu_{0}, \xi \rangle  + \langle d^c\varphi, \xi \rangle, \qquad\text{ with } \qquad
    \langle d^c\varphi,\xi  \rangle= \sum_{a=1}^r \langle d^c\varphi(\xi_a)\xi_a^*,\xi  \rangle,
\end{equation*}
where  $(\xi_a)_{a=1,\cdots,r}$ is a basis of the Lie algebra $\mathfrak{t}$ of $\mathbb{T}$ and $(\xi^*_a)_{a=1,\cdots,r}$ is its dual basis.

For a given positive weight function $\v \in {C}^{\infty}(P,\R_{>0})$, we define \textit{the $\v$-weighted Ricci form} by

\begin{equation}{\label{v:ricci}}
    \Ric_\v(\omega_\varphi) := \Ric(\omega_\varphi) -\frac{1}{2}dd^c \log \v(\mu_\varphi).
\end{equation}
Then the $\v$-weighted scalar curvature of a $\T$-invariant K\"ahler metric $\omega_\varphi \in [\omega_0]$ is defined by

\begin{equation}{\label{def:scalv2}}
 {\Scal_\v(\omega_\varphi)}:= 2\v(\mu_\varphi) \Lambda_{\varphi,\v}(\Ric_\v(\omega_\varphi)),
\end{equation}

\noindent where $\Lambda_{\varphi,\v}$ is the $\v$-weighted trace with respect to $\omega_\varphi$ (see Appendix \ref{a:weighted:trace} for more details about this operator). This definition is equivalent to the one introduced in \cite{Lah19} (see Lemma \ref{l:scalv}), but we do prefer the above more compact expression. 

Given a second weight $\w \in {C}^{\infty}(P,\R)$, a $\T$-invariant K\"ahler metric $\omega_\varphi$ is called \emph{$(\v,\w)$-weighted cscK} if its $\v$-weighted scalar curvature satifies

\begin{equation}{\label{wcsck}}
    \Scal_\v(\omega_\varphi)=\w(\mu_\varphi).
\end{equation}

The significance of \eqref{wcsck} in relation to various geometric conditions is thoroughly examined in \cite{Lah19} (see also \cite{AJL} for $\v$-solitons).  However, we shall mention a few specific cases below:

\begin{enumerate}
    \item[(a)] for $\v = 1$ and $\w=c$ is a constant, \eqref{wcsck} corresponds to the classical cscK equation;
    \item[(b)]  for $\v = 1$ and $\w=\ell_{\ext}$, where $\ell_{\ext}$ is the affine extremal function, the weighted cscK metrics are Calabi's extremal K\"ahler metrics \cite{EC82};
    \item[(c)] for $\v>0$ is any smooth function and $\w=2\v(x)(n+ \langle d\log\v(x),x \rangle )$, \eqref{wcsck} corresponds to the weighted $\v$-soliton examined in \cite{AJL, HL}, generalizing the well-studied K\"ahler-Ricci solitons (see e.g. \cite{BW, Blo, TZ, TZ02});
    \item[(d)] for $\v$ and $\w$ polynomials, then $(\v,\w)$-cscK metrics on $X$ correspond to Calabi's extremal K\"ahler metrics on the total space of an holomorphic fibration $Y$ with fiber $X$, called semisimple principal fibration \cite{ACGT11, AJL, Jub23};
    \item[(e)] for $\v=\ell^{-n-1}, \w=a \ell^{-n-2}$,  $\alpha=c_1(L)$, where $\ell$ is a positive affine-linear function on the polytope, $a$ is a constant,  and $L$ is  a polarization of $X$, \eqref{wcsck} describes  a scalar flat cone K\"ahler metric on the affine cone $(L^{-1})^{\times}$ polarized by the lift of $\xi = d\ell$ to $L^{-1}$ via $\ell$ \cite{AC,ACL}.
\end{enumerate}

\subsection{Main results and strategy of the proof} 

The starting point is to re-write the $4$th order non-linear PDE in \eqref{wcsck} as a system of two linear elliptic PDEs:

 \begin{equation}{\label{weightedsystem:intro}}
     \begin{cases}
     F &= \log\left( \v(\mu_\varphi) \frac{\omega_\varphi^n}{\omega^n_0}\right) \\
    \Delta_{\varphi,\v} F  &=  - \frac{\w ( \mu_\varphi)}{\v(\mu_\varphi)} + 2 \Lambda_{\varphi,\v} ( \Ric(\omega_0 )).
\end{cases}
\end{equation}

\noindent The equivalence between \eqref{wcsck} and \eqref{weightedsystem:intro} is established in Proposition \ref{p:weighted:syst}. \\
In the above equation $\Delta_{\varphi, \v}$ is the \textit{$\v$-weighted Laplacian} with respect to $\omega_{\varphi}$ (we refer to Appendix \ref{a:weighted:lap} for basic properties of this operator).  Note that for $\v=1$ and $\w$ being constant equal to $\frac{n c_1(X) \cdot \{\omega_0\}^{n-1}}{ \{\omega_0\}^{n}}$, we retrieve the system proposed by \cite{CC21a} in the cscK case. 

In the original proof by Chen and Cheng  \cite{CC21a} in the unweighted setting, the authors derive $C^0$ and $C^2$ a priori estimates by establishing an intermediate $C^1$ a priori estimate for the K\"ahler potential $\varphi$. With the $C^0$ and $C^2$ estimates established, higher order estimates are subsequently obtained through standard regularity results for complex Monge-Ampère equations and linear elliptic operators. 
It is important to emphasize that their proof of the $C^0$ estimate \cite[Theorem 5.1]{CC21a} relies on the Alexandroff maximum principle for the real Monge-Amp\`ere operator.

In this paper, following \cite{DD21}, we present a proof of the \(C^0\) and \(C^2\) estimates that bypasses the need for an intermediate $C^1$ estimate. However, it is important to note that the ``weighted" case introduces significantly more challenges and subtleties due to the presence of the weighted operators $\Delta_{\varphi, \v}$ and $\Lambda_{\varphi, \v}$ in \eqref{weightedsystem:intro}. In practice, many additional terms require (uniform) bounds when establishing a priori estimates. We address these terms with detailed explanations to ensure clarity in the presentation. In particular the $C^2$-estimate holds under the assumption that the weight $\v$ is log-concave (i.e. $\log \v$ is a concave function). We are not able to remove such an assumption at this moment. But we do emphasise that in all (but the last one) geometric applications listed above, such assumption is satisfied. 

Our first main result is the $C^0$-estimate, which does not require any condition on the weight:
\begin{customthm}{A}\label{thm:intro:c0}
The functions $\varphi$ and $F$ are uniformly bounded by a constant that depends only on $n, \omega_0$, $\v$, $\w$ and the entropy $\Ent(\varphi):= \int_X \log\left(\frac{\omega_\varphi^n}{\omega_0^n}\right)\omega_\varphi^n$.
\end{customthm}

The proof of the $C^0$-estimate is an extension of the approach proposed in \cite{DD21}, which crucially employs pluripotential theory.

From this, we then get a $C^2$-estimate:

\begin{customthm}{B}\label{thm:intro:c2}
Assume $\v$ is log-concave.
There exists a positive constant $C>0$ depending only on $\omega_0$, $\v$, $\w$, $\|F\|_{C^0}$ and $\| \varphi \|_{C^0}$ such that

\begin{equation*}
\max_X\left(|dF|^2_\varphi + \Lambda_0(\omega_\varphi)\right) \leq C
\end{equation*}
where $\Lambda_0$ is the trace with respect to $\omega_0$.
\end{customthm}
Two crucial intermediate steps consist in establishing a priori integral $L^p$-estimates ($p\geq 1$) for the Laplacian $\Delta_{\varphi}$ (Theorem \ref{thm:integral}) and a lower bound for the weighted Laplacian $\Delta_{\varphi,\v} u$ for $u := e^{\frac{F}{2}} |dF|_\varphi^2 + K \Lambda_0(\omega_\varphi)$ (Proposition \ref{p:weightedlap:u}).

\smallskip






Noting that $\Delta_{\varphi,\v}$ is elliptic (see Lemma \ref{lemma weighted Laplacian}), the higher-order estimates can be derived using standard arguments of elliptic and complex Monge-Amp\`ere theories \cite{Au76, LU} similar to those used in the usual cscK problem (cf \cite[page 3]{CC21a}).

\subsection{Notations} In the text we will denote by  $\Delta_\varphi, \Lambda_\varphi$ the Laplacian and the trace with respect to $\omega_\varphi$. The scalar curvature $\Scal$ is defined as twice the trace of the Ricci form, i.e. $\Scal (\omega_\varphi)= 2 \Lambda_\varphi( \Ric(\omega_\varphi))$ where the Ricci form is defined by $\Ric(\omega_\varphi):=\ric_\varphi(J\cdot,\cdot)$ with $\ric_\varphi$ the Ricci symmetric 2-tensor of the underlying riemannian metric $g_\varphi$. We denote by $\nabla^\varphi$ the gradient with respect to $g_\varphi$. 

\subsection{Organisation of the paper} In Section \ref{s:plury} we introduce notations and results which will be used throughout the paper: we provide definitions and references to pluripotential theory. Section \ref{s:wsystem} is dedicated to establish the equivalence between the fourth-order PDE \eqref{wcsck} and the system of elliptic PDE's \eqref{weightedsystem:intro}. 
The proof of the \(C^0\)-estimate, stated in Theorem \ref{thm:intro:c0}, is presented in Section \ref{s:c0-estimates}. 
In Section \ref{s:integral:c2:esti}, we prove the \(L^p\)-estimates needed for the proof of Theorem \ref{thm:intro:c2} which can be found in Section \ref{s:c2:esti}.

\section*{Acknowledgements}
The first author is supported by the project SiGMA ANR-22-ERCS-0004-02. The authors are grateful to Vestislav Apostolov for valuable suggestions. We thank Tat Dat Tô and Chung-Ming Pan for useful comments on earlier version of the paper. The author also thank S\`ebastien Boucksom and Mattias Jonsson for useful discussions.

We would like to acknowledge Jiyuan Han and Yaxiong Liu for bringing to our attention the overlap between our project and their work \cite{HanLiu}. Despite these similarities, our approach introduces some distinct methodologies. Notably, our $C^0$-estimate leverages pluripotential theory, and we circumvent the need for a gradient estimate by directly deriving the $C^2$-estimate from the $C^0$ one.

Par of this material is based upon work done while the first author was supported by the National ScienceFoundation under Grant No. DMS-1928930, while the author was in
residence at the Simons Laufer Mathematical Sciences Institute
(formerly MSRI) in Berkeley, California, during the Fall 2024 semester.

\section{Preliminaries in pluripotential theory}{\label{s:plury}}

We recall below some ingredients from pluripotential theory that are going to be crucial in what follows in order to establish uniform $C^0$-estimates.
The first is a powerful integrability result which is known as a uniform version of Skoda’s integrability theorem. We introduce
 $$\nu_{\omega_0} :=\sup_{u,x} \nu(u,x), \quad x\in X,\; u\in \PSH(X, \omega_0), $$ where $\nu(u, x)$ denotes the Lelong number of $u$ at $x$. We note that from the proof of \cite[Lemma 8.10]{GZ17} one can deduce that $\nu_{\omega_0} \geq 1$.
 
\begin{theorem}\label{thm: Skoda} 
Let $c< 2 \nu_{\omega_0}^{-1} $. Then there exists a uniform constant $C>0$ such that for all $u\in \PSH(X, \omega_0)$ with $\sup_X u=0$ we have
$$\int_X e^{-c u} \omega_0^{n} \leq C.$$
\end{theorem}
We refer to \cite[Theorem 8.11]{GZ17} for a proof. The following result is due to Ko{\l}odziej \cite{Kolo}:
 \begin{theorem}\label{thm: Kol} 
 Assume $\omega_u^n=f \omega_0^n$ with $f\in L^p$ for some $p>1$. Then there exists $C>0$ depending only on  $\omega, n, \|f\|_{L^p}$ such that 
 $$ \Osc_X u\leq C$$
 \end{theorem}
\noindent Here $L^p:=L^p\left(\omega_0^{n} \right)$. In the following we specify the reference volume form in the notation of the $L^p$ norms only if is different from the standard one.

At last, we recall \cite[Theorem 3.3]{DDL4}, that can be viewed as a generalization of Ko{\l}odziej's theorem:
\begin{theorem}\label{thm: uniform estimate} 
	Fix $a\in [0,1),A>0$, $\chi \in \PSH(X,\omega_0)$ and  $0\leq f \in L^p$ for some $p>1$. Assume that  $u\in \PSH(X,\omega_0)$, normalized by $\sup_X u=0$, satisfies 
	\begin{equation*}
		\label{eq: volume cap domination 0}
		\omega_u^{n}  \leq f\omega_0^{n}  + a\omega_{\chi}^{n} .
	\end{equation*}
	Assume also that  
	\begin{equation}
		\label{eq: volume cap domination}
		\int_E f\omega_0^{n}  \leq A [\capa_{\chi}(E)]^2,
	\end{equation}
	for every Borel subset $E\subset X$. If $P[u]$ is less singular than $\chi$ (i.e. $\chi \leq P[u]+C$, for some $C>0$) then 
$$\chi -\sup_X \chi- C\Big(\|f\|_{L^p},p,(1-a)^{-1},A\Big) \leq u.$$
\end{theorem}
It is worth it to mention that such a result is stated and proved in a much more general version in \cite{DDL4} to which we refer for a proof. Here $\capa_{\chi}(E)$ is the $\chi$-relative capacity of $E$ and it is defined as 
\begin{equation*}\label{eq: capphi_def}
 \textup{Cap}_\chi(E) := \sup \{ 
 \int_E \omega_{u}^{n}  \; :\; u \in \PSH(X,\omega_0),\, \chi-1\leq u \leq \chi 
\} 
\end{equation*}
and 
$$P[u] = \left(\sup \{v \in 
\textup{PSH}(X,\omega_0), \ v \leq 0 \textup{ and } v \leq u+C,   \textup{ for some }  C>0 \} \right)^*, $$
where $*$ denotes the upper semi-continuous regularization.
For later purposes we mention that $P[u]=0$ if and only if $u$ is such that $\int_X \omega_u^n =V$ (\cite[Theorem 1.3]{DDL2}).

\section{Weighted cscK equation as a system of elliptic PDEs}{\label{s:wsystem}}
The goal of this section is to write the weighted cscK equation \eqref{wcsck} as a system of two elliptic PDEs. We start by showing the identity \eqref{def:scalv2} for the weighted scalar curvature $\Scal_\v(\omega_\varphi)$ defined in \cite{Lah19} by
\begin{equation}{\label{def:scalv}}
     \Scal_\v(\omega_\varphi)= \vv \Scal(\omega_\varphi) - 2 \Delta_\varphi \vv + \sum_{a,b=1}^r \v_{,ab}(\mu_\varphi) g_\varphi(\xi_a, \xi_b).
\end{equation}

We develop the second term from the RHS of \eqref{def:scalv}:

\begin{equation}{\label{comp:1}}
\begin{split}
    \Delta_\varphi \v(\mu_\varphi) =& \Lambda_\varphi d d^c \v(\mu_\varphi) \\
    =& \Lambda_\varphi d \left( \sum_{a=1}^r \v_{,a}(\mu_\varphi) d^c \mu_\varphi^a \right) \\
    =& \Lambda_\varphi \left( \sum_{a,b=1}^r \v_{,a}(\mu_\varphi) dd^c \mu_\varphi^a + \v_{,ab}(\mu_\varphi) d \mu_\varphi^a \wedge d^c \mu_\varphi^b\right) \\
      =& \sum_{a,b=1}^r \v_{,a}(\mu_\varphi) \Delta_\varphi \mu_\varphi^a  -\v_{,ab}(\mu_\varphi) \Lambda_\varphi \left(   \omega_\varphi( \xi_a, \cdot ) \wedge   \omega_\varphi( \xi_b, J \cdot ) \right),
     \end{split}
     \end{equation}
where $\mu_\varphi^a:= \langle \mu_\varphi,\xi_a \rangle$  and we use that $\omega_{\varphi}(\xi_a, \cdot)= -d\mu^a_\varphi$ and $\omega_{\varphi}(\xi_b, J \cdot)= d^c\mu^b_\varphi$ for the last equality. We continue the computation     
\begin{equation*}
\begin{split}
    \Delta_\varphi \v(\mu_\varphi)    
     =&\sum_{a,b=1}^r \v_{,a}(\mu_\varphi) \Delta_\varphi \mu_\varphi^a  - \v_{,ab}(\mu_\varphi) \Lambda_\varphi \left(   g_\varphi( J \xi_a,  \cdot ) \wedge   g_\varphi( \xi_b,  \cdot ) \right)  \\
     =&\sum_{a,b=1}^r \left(\v_{,a}(\mu_\varphi) \Delta_\varphi \mu_\varphi^a  - \v_{,ab}(\mu_\varphi) \sum_{k=1}^{2n}   g_\varphi( J \xi_a,  e_k )    g_\varphi( \xi_b,  Je_k )\right) \\
     =&\sum_{a,b=1}^r \v_{,a} (\mu_\varphi) \Delta_\varphi \mu_\varphi^a  + \v_{,ab}(\mu_\varphi)    g_\varphi(  \xi_a,  \xi_b ),
\end{split}
\end{equation*}
\noindent where $(e_k)^{2n}_{k=1}$ is a local orthonormal frame w.r.t. $g_\varphi$. We showed

\begin{equation}\label{comp:deltav}
    - \Delta_\varphi \v(\mu_\varphi) = - \sum_{a,b=1}^r \v_{,a}(\mu_\varphi) \Delta_\varphi \mu_\varphi^a  - \v_{,ab}(\mu_\varphi)    g_\varphi(  \xi_a,  \xi_b ).
\end{equation}
For later purposes, let us observe that the first 4th lines of the above computation give
\begin{eqnarray}\label{eq: lap v}
 \nonumber \Delta_0 \v(\mu_\varphi)&=& \sum_{a,b=1}^r \v_{,a}(\mu_\varphi) \Delta_0 \mu_\varphi^a  +\v_{,ab}(\mu_\varphi) \Lambda_0 \left(   d \mu_\varphi^a \wedge d^c \mu_\varphi^b \right)\\
 &=& \sum_{a,b=1}^r \v_{,a}(\mu_\varphi) \Delta_0 \mu_\varphi^a  +\v_{,ab}(\mu_\varphi) g_0(d \mu_\varphi^a, d \mu_\varphi^b),
 \end{eqnarray}
where the last identity follows from a general fact, i.e. that for smooth functions $f_1, f_2$,
\begin{equation}\label{eq:standard}
g_0 \left(   d f_1, d f_2 \right) = \Lambda_0 \left(   d f_1 \wedge d^c f_2 \right).
\end{equation}

\begin{lemma}
The scalar curvature has the alternative expression

    \begin{equation}{\label{formul:scal1}}
\begin{split}
        \Scal(\omega_\varphi)=& 2 \Lambda_{\varphi}(\Ric(\omega_0)) -  \Delta_\varphi F +  \frac{1}{\vv}\left(  \sum_{a,b=1}^r \v_{,a}(\mu_\varphi) \Delta_\varphi \mu_\varphi^a + \v_{,ab}(\mu_\varphi)    g_\varphi(  \xi_a,  \xi_b ) \right)  \\
        & - \frac{1}{\vv^2} \left( \sum_{a,b=1}^r\v_{,a}(\mu_\varphi) \v_{,b}(\mu_\varphi) g_\varphi( \xi_a, \xi_b) \right),
\end{split}        
\end{equation}

\noindent where $F:=\log\left( \v(\mu_\varphi) \frac{\omega_\varphi^n}{\omega^n_0}\right)$.
\end{lemma}

\begin{proof}
By definition     
\begin{equation}\label{comp:scal:1}
    \begin{split}
        \Scal(\omega_\varphi)= 2 \Lambda_{\varphi}(\Ric(\omega_0)) -  \Delta_\varphi \log \frac{\omega_\varphi^n}{\omega^n_0}.
    \end{split}
\end{equation}
Observe that there is a factor $2$ since, with our notations the scalar curvature is twice the trace. Also, by definition of $F$

\begin{equation*}
     \Delta_\varphi \log \frac{\omega^n}{\omega^n_0} = \Delta_\varphi F - \Delta_\varphi \log \v(\mu_\varphi).
\end{equation*}
We develop 

\begin{equation*}
    \begin{split}
        \Delta_\varphi \log \v(\mu_\varphi) =& \Lambda_\varphi d d^c \log \v(\mu_\varphi) \\
          =& \frac{1}{\vv} \Delta_\varphi \v(\mu_\varphi) - \Lambda_\varphi  \left( \frac{1}{\vv^2} d \vv \wedge d^c
         \vv \right) \\
          =&  \frac{1}{\vv}\Delta_\varphi \v(\mu_\varphi) - \frac{1}{\vv^2} \Lambda_\varphi  \left( \sum_{a,b=1}^r\v_{,a} (\mu_\varphi) \v_{,b} (\mu_\varphi)d\mu_\varphi^a \wedge d^c \mu_\varphi^b \right) \\
          =&  \frac{1}{\vv}\Delta_\varphi \v(\mu_\varphi) - \frac{1}{\vv^2} \left( \sum_{a,b=1}^r\v_{,a} (\mu_\varphi) \v_{,b} (\mu_\varphi)g_\varphi( \xi_a, \xi_b) \right).
        \end{split}
\end{equation*}
The last equality follows from the same computation as in \eqref{comp:1}. Then, using \eqref{comp:deltav},

\begin{equation*}
\begin{split}
     \Delta_\varphi \log \frac{\omega_{\varphi}^n}{\omega^n_0}=& \Delta_\varphi F -  \frac{1}{\vv}\Delta_\varphi \v(\mu_\varphi) + \frac{1}{\vv^2} \left( \sum_{a,b=1}^r\v_{,a}(\mu_\varphi) \v_{,b}(\mu_\varphi) g_\varphi( \xi_a, \xi_b) \right) \\
     =& \Delta_\varphi F -  \frac{1}{\vv} \left(  \sum_{a,b=1}^r \v_{,a}(\mu_\varphi) \Delta_\varphi \mu_\varphi^a + \v_{,ab}(\mu_\varphi)    g_\varphi(  \xi_a,  \xi_b ) \right)  \\
     &+ \frac{1}{\vv^2} \left( \sum_{a,b=1}^r\v_{,a}(\mu_\varphi) \v_{,b}(\mu_\varphi) g_\varphi( \xi_a, \xi_b) \right).
\end{split}
\end{equation*}
Injecting in \eqref{comp:scal:1}, we get the result.
\end{proof}

\begin{lemma}{\label{l:scalv}}
    The weighted scalar curvature has the alternative expression
\begin{equation*}
    \Scal_\v(\omega_\varphi) = 2\v(\mu_\varphi) \Lambda_{\varphi,\v}(\Ric_\v(\omega_\varphi))
\end{equation*}

\noindent where  $\Lambda_{\varphi,\v}$ is defined in \eqref{def:weighted:trace} and $\Ric_\v$ is the $\v$-weighted Ricci from, defined in \eqref{v:ricci}.
\end{lemma}

\begin{proof}

 We first re-write the term $\sum_{a=1}^r\v_{,a}(\mu_\varphi) \Delta_\varphi \mu_\varphi^a$. For $\xi_b \in \mathfrak{t}$, we have

\begin{equation}{\label{comp:dcpsi}}
    \begin{split}
d^c F (\xi_b) = & \frac{d^c\vv(\xi_b)}{\vv} + d^c \log \left( \frac{\omega_\varphi^n}{\omega_0^n} \right)(\xi_b)  \\
=& \frac{\sum_{a=1}^r \v_{,a}(\mu_\varphi)d^c \mu_\varphi^a (\xi_b)}{\vv} + d^c \log \left( \frac{\omega_\varphi^n}{\omega_0^n} \right)(\xi_b) \\
=& \frac{\sum_{a=1}^r \v_{,a}(\mu_\varphi)g_\varphi(\xi_a,\xi_b)}{\vv} + d^c \log \left( \frac{\omega_\varphi^n}{\omega_0^n} \right)(\xi_b). \\
    \end{split}
\end{equation}
Also, using Cartan's formula and the fact that $\xi_b\lrcorner\omega_\varphi=-d\mu_\varphi^b$ we deduce $\mathcal{L}_{J\xi_b}\omega_\varphi=-dd^c \mu_\varphi^b$, where $\mathcal{L}_{J\xi_b}$ denotes the Lie derivative in the direction $J\xi_b$. And the same holds for $\omega_0$. It follows from arguments in the proof of \cite[Lemma 5]{Lah19}
\begin{equation}{\label{comp:dcf}}
d^c \log\left( \frac{\omega_\varphi^n}{\omega^n_0} \right)(\xi_b) =  \Delta_\varphi \mu_\varphi^b - \Delta_0 \mu_0^b ,
\end{equation}

\noindent Multiplying \eqref{comp:dcf} by  $\v_{,b}(\mu_\varphi)$, summing over $b=1,\cdots,r$ and using \eqref{comp:dcpsi} we obtain

\begin{equation}{\label{comp:vdeltamu}}
    \begin{split}
\sum_{b=1}^r\v_{,b}(\mu_\varphi) \Delta_\varphi \mu_\varphi^b = &\sum_{b=1}^r\v_{,b}(\mu_\varphi) \Delta_0 \mu_0^b + \v_{,b}(\mu_\varphi)  d^c \log\left( \frac{\omega_\varphi^n}{\omega^n_0} \right)(\xi_b) \\
=& \sum_{b=1}^r\v_{,b}(\mu_\varphi) \Delta_0 \mu_0^b + \v_{,b}(\mu_\varphi) d^c F(\xi_b) - \sum_{a,b=1}^r   \frac{ \v_{,a}(\mu_\varphi) \v_{,b}(\mu_\varphi)g_\varphi(\xi_a,\xi_b)}{\vv}.
    \end{split}
\end{equation}

\noindent Injecting in \eqref{formul:scal1}, we find 

\begin{eqnarray}{\label{comp:scal2}}
\nonumber \Scal(\omega_\varphi)&=& 2 \Lambda_{\varphi}(\Ric(\omega_0)) -  \Delta_\varphi F + \frac{1}{\vv}\sum_{a,b=1}^r  \v_{,ab}(\mu_\varphi)    g_\varphi(  \xi_a,  \xi_b )\\
 \nonumber       &&- \frac{1}{\vv^2} \left(\sum_{a,b=1}^r\v_{,a}(\mu_\varphi) \v_{,b}(\mu_\varphi) g_\varphi( \xi_a, \xi_b) \right)  \\
   \nonumber        &&+  \frac{1}{\vv}\left(  \sum_{b=1}^r\v_{,b}(\mu_\varphi) \Delta_0 \mu_0^b + \v_{,b}(\mu_\varphi) d^c F(\xi_b) - \sum_{a,b=1}^r   \frac{ \v_{,a}(\mu_\varphi) \v_{,b}(\mu_\varphi)g_\varphi(\xi_a,\xi_b)}{\vv} \right) \\
  &=& 2 \Lambda_{\varphi}(\Ric(\omega_0)) -  \Delta_\varphi F - \frac{2}{\vv^2} \left( \sum_{a,b=1}^r\v_{,a}(\mu_\varphi) \v_{,b}(\mu_\varphi) g_\varphi( \xi_a, \xi_b) \right)  \\
  \nonumber&&  +   \frac{1}{\vv}\left( \sum_{b=1}^r\v_{,b}(\mu_\varphi) \Delta_0 \mu_0^b + \v_{,b}(\mu_\varphi) d^c F(\xi_b) + \sum_{a,b=1}^r  \v_{,ab}(\mu_\varphi)    g_\varphi(  \xi_a,  \xi_b) \right) 
\end{eqnarray}

We conclude by injecting \eqref{comp:deltav} and \eqref{comp:scal2} in \eqref{def:scalv}:

\begin{eqnarray}{\label{comp:scalv}}
\nonumber     \Scal_\v(\omega_\varphi)&=&
   2 \vv \Lambda_{\varphi}(\Ric(\omega_0)) - \vv \Delta_\varphi F   + \sum_{b=1}^r\v_{,b}(\mu_\varphi) \Delta_0 \mu_0^b + \v_{,b}(\mu_\varphi) d^c F(\xi_b) \\
  \nonumber   &&+ \sum_{a,b=1}^r  \v_{,ab}(\mu_\varphi)    g_\varphi(  \xi_a,  \xi_b) - \frac{2}{\vv} \left( \sum_{a,b=1}^r\v_{,a}(\mu_\varphi) \v_{,b}(\mu_\varphi) g_\varphi( \xi_a, \xi_b) \right)  \\
   \nonumber   && -  2 \sum_{a,b=1}^r \v_{,a}(\mu_\varphi) \Delta_\varphi \mu_\varphi^a  - 2\sum_{a,b=1}^r \v_{,ab}(\mu_\varphi)    g_\varphi(  \xi_a,  \xi_b ) + \sum_{a,b=1}^r \v_{,ab}(\mu_\varphi) g_\varphi(\xi_a, \xi_b)   \\
\nonumber   &=& 2 \vv  \Lambda_{\varphi}(\Ric(\omega_0)) - \vv \Delta_\varphi F   + \sum_{b=1}^r\v_{,b}(\mu_\varphi) \Delta_0 \mu_0^b + \v_{,b}(\mu_\varphi) d^c F(\xi_b) \\
 \nonumber     && -  2 \sum_{a=1}^r \v_{,a}(\mu_\varphi) \Delta_\varphi \mu_\varphi^a - \frac{2}{\vv} \left( \sum_{a,b=1}^r\v_{,a}(\mu_\varphi) \v_{,b}(\mu_\varphi) g_\varphi( \xi_a, \xi_b) \right),
 \end{eqnarray}
 where in the last equality we used \eqref{comp:vdeltamu}. We continue the computation:

 \begin{eqnarray}
 \nonumber   \Scal_\v(\omega_\varphi)   &=& \vv 2 \Lambda_{\varphi}(\Ric(\omega_0)) -   \vv \Delta_\varphi F  - \sum_{b=1}^r \v_{,b}(\mu_\varphi) d^c F(\xi_b) - \v_{,b}(\mu_\varphi) \Delta_0 \mu_0^b \\  
  \nonumber   & =& \vv 2 \Lambda_{\varphi}(\Ric(\omega_0)) +   d^*(\vv dF) - \sum_{b=1}^r\v_{,b}(\mu_\varphi) \Delta_0 \mu_0^b \\
    \nonumber    &=&  \v(\mu_\varphi)\left(- \Delta_{\varphi,\v} F + 2 \Lambda_{\varphi,\v}  \Ric(\omega_0 )\right) \\
 \nonumber       &=&2\v(\mu_\varphi) \Lambda_{\varphi,\v}(\Ric_\v(\omega_\varphi)),
\end{eqnarray}  
 In the last line we use \eqref{weighted trace} for the expression of $\Lambda_{\varphi,\v}  \Ric(\omega_0 )$, the definition of the $\v$-weighted Laplacian in \eqref{def weighted lapl} and that $\Delta_{\varphi,\v}f=\Lambda_{\varphi, \v}dd^c f$ (see Appendix \ref{a:weighted:lap}). 
\end{proof}


We then conclude this section:

\begin{prop}{\label{p:weighted:syst}}
A $\T$-invariant K\"ahler metric is a weighted cscK metric if and only if it is solution of the following system of elliptic PDEs
 
 \begin{equation}\label{weighted system}
     \begin{cases}
     F &= \log\left( \v(\mu_\varphi) \frac{\omega_\varphi^n}{\omega^n_0}\right) \\
    \Delta_{\varphi,\v} F  &=  - \frac{\w ( \mu_\varphi)}{\v(\mu_\varphi)} + 2 \Lambda_{\varphi,\v} ( \Ric(\omega_0 )),
     \end{cases}
\end{equation}\noindent where $\Delta_{\varphi,\v}$ and $\Lambda_{\varphi,\v}$ are introduced in Appendix \ref{s:appen}.
\end{prop}

\begin{proof}
From the end of the proof of Lemma \ref{l:scalv} we have that $\frac{1}{\v(\mu_\varphi)}\Scal_\v(\omega_\varphi)=- \Delta_{\varphi,\v} F + 2 \Lambda_{\varphi,\v}( \Ric(\omega_0))$. Then $\Scal_\v(\omega_\varphi)=\w(\mu_\varphi)$ if and only if the second equation in \eqref{weighted system} is satisfied.
\end{proof}

\section{$\mathcal{C}^0$-estimates}{\label{s:c0-estimates}}
The goal of this section is to prove a priori $C^0$-estimates for $\varphi$ and $F$, solutions of \eqref{weighted system} where $\varphi$ is normalized such that $\sup_X \varphi=0$. 

First, observe that since $P$ is compact and $\v$, $\w$ are smooth, there exist positive constants $\eta$, $L$, $\nu$, $M$ such that

\begin{equation}{\label{bound:v}}
\begin{split}
&0 <\eta \leq \v \leq 1 \leq L<+\infty \text{, }  \quad \quad -\nu \leq \w \leq M, \\
&-\eta \leq \v_{,j} \leq L<+\infty,\quad   \quad -\nu \leq \w_{,j}\leq M.
\end{split}
\end{equation}
Let $\psi$ be the unique solution of the Monge-Amp\`ere equation
$$\omega_\psi^{[n]}= b^{-1} \sqrt{F^2+1}  \omega_\varphi^{[n]}, \quad \sup_X \psi = 0, $$
\noindent $b:= \int_X \sqrt{F^2+1}\,\omega_\varphi^{[n]}$. We normalize $\omega_0$ such that 

\begin{equation}{\label{norma-omega0}}
    \int_X \omega_0^{[n]}=1,
\end{equation}
where we recall that $\omega_0^{[k]}:= \frac{\omega_0^k}{k!}$, for $1 \leq k \leq n$. Set $m_\v:=\int_X e^F \omega_0^{[n]}= \int_X \v(\mu_\varphi) \omega_\varphi^{[n]}$, so that $\frac{\v(\mu_\varphi) \omega_\varphi^{[n]}}{m_\v}$ is a probability measure. \\
Observe that, since $F^2+1\leq 2 F^2$ on $\{ F\geq 1\}$,

\begin{equation*}
\begin{split}
0<b= &\int_{\{F<1\}} \v(\mu_\varphi)^{-1} e^F\sqrt{F^2+1}\,\omega_0^{[n]} + \int_{\{F\geq 1\}}\v(\mu_\varphi )^{-1}  e^F\sqrt{F^2+1}\,\omega_0^{[n]} \\\leq &\sqrt{2}  \eta^{-1}(e + \Ent_{\v}(\varphi))
\end{split}
\end{equation*}
where $\Ent_\v(\varphi):= \int_X \log\left( \frac{\v(\mu_\varphi) \omega_\varphi^n}{m_\v \omega_0^n} \right) \frac{\v(\mu_\varphi) \omega_\varphi^{[n]}}{{m_\v}}$ denotes the \emph{weighted entropy}, between the probability measures $\frac{\v(\mu_\varphi) \omega_\varphi^n}{m_\v}$ and $\omega_0^n$. By Jensen inequality $\Ent_\v(\varphi) \geq 0$.\\
We observe that the latter is comparable to the classical entropy $\Ent(\varphi):=\Ent_1(\varphi)$:

\begin{lemma}
 For any weight $\v>0$, the weighted entropy $\Ent_\v$ is bounded if and only if the entropy $\Ent$ is:
 \begin{equation*}
     \Ent_\v(\varphi) < \infty \hspace*{0.3cm} \Leftrightarrow  \hspace*{0.3cm} \Ent(\varphi) < \infty,
 \end{equation*}
for any $\varphi \in \mathcal{K}(X,\omega_0)^\T$.
 \end{lemma}
\begin{proof}
Suppose first that the entropy $\Ent(\varphi)$ is bounded. For simplicity, we write $g:=\e^F$ and $f:=\frac{g}{\v(\mu_\varphi)}$. Then $g=\frac{\v(\mu_\varphi)\omega_\varphi^{[n]}}{\omega_0^{[n]}}$. Since $\v \log \v \leq L \log L$ and $- e^{-1}\leq x\log x \leq 0$ for any $x\in [0, 1]$, we find
\begin{equation*}
    \begin{split}
        \Ent_\v(\varphi)= &\frac{1}{m_\v}\int_X g \log g \,\omega_0^{[n]} -\log m_\v\\
        = & \frac{1}{m_\v} \int_X \v(\mu_\varphi) f \log (\v(\mu_\varphi) f) \,\omega_0^{[n]} -\log m_\v\\
        =& \frac{1}{m_\v} \int_X \v(\mu_\varphi) \log (\v(\mu_\varphi) ) f \,\omega_0^{[n]} +\frac{1}{m_\v} \int_X \v(\mu_\varphi) f \log f \,\omega_0^{[n]} -\log m_\v\\
        \leq & \frac{L \log L}{m_\v}  \int_X  \omega_\varphi^{[n]} + \frac{\eta}{m_\v}\int_{\{f \leq 1\}} f \log f \,\omega_0^{[n]}+ \frac{L}{m_\v}\int_{\{f> 1\}} f \log f \,\omega_0^{[n]}  -\log m_\v\\
        \leq  & \frac{L \log L}{m_\v} + \frac{L}{m_\v} \int_{X} f \log f \,\omega_0^{[n]} + \frac{L}{m_\v}\int_{\{f\leq 1\}} f (-\log f ) \,\omega_0^{[n]}-\log m_\v \\
        \leq  & \frac{L \log L}{m_\v} + \frac{L}{m_\v} \Ent(\varphi) + \frac{L}{m_\v} e^{-1}-\log m_\v 
        \end{split}
\end{equation*}
where we use \eqref{norma-omega0} to ensure that $\int_X  \omega_\varphi^{[n]} =1$ and $L$ is defined in \eqref{bound:v}. The converse is obtain via similar computations.
\end{proof}

Therefore, if $\Ent(\varphi)$ is uniformly bounded, so is $b$.\\
 
Let $A_0>0$ such that $-A_0 \omega_0 \leq \Ric(\omega_0) \leq A_0 \omega_0 $. We start with the following key observation:

\begin{equation}{\label{bound-trace-weight}}
\begin{split}
|\Lambda_{\v, \varphi} \Ric(\omega_0)|&=\left|\Lambda_{\varphi}(\Ric(\omega_0)) -\frac{1}{2\v(\mu_\varphi)}\sum_{a=1}^r \v_{,a}(\mu_\varphi) \Delta_0 \mu_0^a \right|\\
 &\leq A_0 \Lambda_{ \varphi} (\omega_0) +\sup_X \left| \frac{1}{2\v(\mu_\varphi)}\sum_{a=1}^r \v_{,a}(\mu_\varphi) \Delta_0 \mu_0^a \right|\\
 &\leq A_0 \Lambda_{ \varphi} (\omega_0) +C.
\end{split}
\end{equation}
It is important to note that the constant $C$ above does not depend on $\varphi$ by \eqref{bound:v}. \\
Following \cite{DD21} we show:

\begin{theorem}\label{step 1 C0 weighted}
Given $\varepsilon\in(0,1)$, there exists $C=C(\varepsilon, n, \omega_0, P, \w, \v, b)$ such that
$$F+\varepsilon \psi- A\varphi\leq C,$$
where $A>0$ is a uniform constant depending only on the lower bound of the Ricci curvature.
\end{theorem}

\begin{proof}
We let $H:=F+\varepsilon \psi-A \varphi$ with $A=2A_0+1$. 
Then
\begin{equation*}
\begin{split}
\Delta_{\varphi, \v} H=&\Delta_{\varphi, \v} F+\varepsilon \Delta_{\varphi, \v} \psi   - A  \Delta_{\varphi, \v} \varphi \\
  =&-\frac{\w\left(\mu_\varphi\right)}{\v(\mu_\varphi)}+2\Lambda_{ \varphi, \v}\Ric(\omega_0) +\varepsilon  \Delta_{\varphi} \psi-A  \Delta_{\varphi} \varphi \\
  &+\frac{1}{\v(\mu_\varphi)}\sum_{a=1}^r \v_{,a}\left(\mu_\varphi\right)\left(\varepsilon d^c(\psi)\left(\xi_a\right)-A d^c \varphi\left(\xi_a\right)\right) \\
=&-\frac{\w\left(\mu_\varphi\right)}{\v(\mu_\varphi)}+2\Lambda_{ \varphi, \v}\Ric(\omega_0) +\varepsilon  \frac{d d^c \psi \wedge \omega_{\varphi}^{[n-1]}}{\omega_\varphi^{[n]}} -  A  \frac{d d^c \varphi \wedge \omega_\varphi^{[n-1]}}{\omega_\varphi^{[n]}} \\ & +\frac{1}{\v(\mu_\varphi)}\sum_{a=1}^r  \v_{,a}\left(\mu_\varphi\right)\left(\varepsilon d^c \psi\left(\xi_a\right)-A d^c \varphi\left(\xi_a\right)\right) \\
\geq &-\frac{\w\left(\mu_\varphi\right)}{\v(\mu_\varphi)}- 2 A_0 \Lambda_{\varphi}(\omega_0) -C +\varepsilon  \frac{d d^c \psi \wedge \omega_{\varphi}^{[n-1]}}{\omega_\varphi^{[n]}}  -  A  \frac{d d^c \varphi \wedge \omega_\varphi^{[n-1]}}{\omega_\varphi^{[n]}}  \\ &  +\sum_{a=1}^r \frac{1}{\v(\mu_\varphi)} \v_{,a}\left(\mu_\varphi\right)\left(\varepsilon d^c \psi\left(\xi_a\right)-A d^c \varphi\left(\xi_a\right)\right).
\end{split}
\end{equation*}

For the inequality, we use \eqref{bound-trace-weight}. We continue the computation
\begin{equation*}
\begin{split}
\Delta_{\varphi, \v} H \geq & -\frac{\w\left(\mu_\varphi\right)}{\v(\mu_\varphi)} -  2 A_0 \Lambda_\varphi (\omega_0)-C  + \varepsilon \frac{\left(\omega_\psi-\omega_0\right) \wedge \omega_{\varphi}^{[n-1]}}{\omega_\varphi^{[n]}} \\
&- A  \frac{\left(\omega_{\varphi}-\omega_0\right) \wedge \omega_{\varphi}^{[n-1]}}{\omega_\varphi^{[n]}}  +\frac{1}{\v(\mu_\varphi)}\sum_{a=1}^r  \v_{,a}\left(\mu_\varphi\right)\left(\varepsilon d^c \psi\left(\xi_a\right)-A d^c \varphi\left(\xi_a\right)\right) \\  
\geq & -\frac{M}{L} - A n - C  + (A - 2A_0 - \varepsilon)\Lambda_\varphi (\omega_0) +\varepsilon  \frac{\omega_\psi \wedge \omega_\varphi^{[n-1]}}{\omega_\varphi^{[n]}}   \\ 
&  +\frac{1}{\v(\mu_\varphi)}\sum_{a=1}^r  \v_{,a}\left(\mu_\varphi\right)\left(\varepsilon d^c \psi\left(\xi_a\right)-A d^c \varphi\left(\xi_a\right)\right) \\
 \geqslant& -C_1 + n b^{-1/n} \varepsilon\left(F^2+1\right)^{\frac{1}{2 n}} + \frac{1}{\v(\mu_\varphi)} \sum_{a=1}^r \v_{,a} \left(\mu_\varphi\right)\left(\varepsilon\left(\mu_\psi^a-\mu_0^a\right)-A\left(\mu_\varphi^a-\mu_0^a\right)\right)\\
  \geqslant& - C_1 + C_2 +n b^{-1/n}\varepsilon\left(F^2+1\right)^{\frac{1}{2 n}} 
\end{split}
\end{equation*}
where $C_1$ depends on $\eqref{bound-trace-weight}$, $C_2:=\inf_X \sum_{a=1}^r \frac{\v_{,a}\left(\mu_\varphi\right)}{\v(\mu_\varphi)} \left(\varepsilon\left(\mu_\psi^a-\mu_0^a\right)-A\left(\mu_\varphi^a-\mu_0^a\right)\right)$. Observe that the constant $C_2$ is indeed independent of $\varphi$ and $\psi$ since the image of $\mu_\varphi$ is the moment polytope $P$ for any $\T$-invariant K\"ahler potential in $[\omega_0]$ (see the introduction and \cite[Lemma 1]{Lah19} for details).

From the second inequality to the third, we use the mixed Monge-Amp\`ere inequality (see \cite[Proposition 1.11]{BEGZ}) ensuring that $\omega_\psi \wedge \omega_\varphi^{n-1} \geq b^{-1/n} (\sqrt{F^2+1})^{1/n}  \,\omega_\varphi^n$. We also used that the moment map $\mu_\varphi$ of a K\"ahler potential $\varphi$ satisfies $\mu_\varphi^a = \mu_0^a - d^c\varphi(\xi^a)$ for any $\xi^a \in \mathfrak{t}$.

By the maximum principle, applied to $H$, we can then infer that at a maximum point $x_0$ we have
$$n b^{-1/n}\varepsilon (F^2+1)^{1/2n} (x_0)\leq C . $$ Thus $F(x_0)\leq C_0$, $C_0=C_0(\varepsilon, A_0, \omega_0, b, \v, \w)$.

\smallskip
We then claim that $$\varepsilon \psi- A\varphi \leq C_3,$$
where $C_3>0$ depends on $\varepsilon$, $A$ and $b$. Let us now prove the claim. First of all we observe that, for any $a, \delta\in(0,1)$ we have either $\sqrt{F^2+1} \geq b/ (a\delta^n)$ or $F\leq \sqrt{F^2+1} \leq b/ (a\delta^n)$; thus

\begin{equation*}
    \begin{split}
\omega_\varphi^{n}=&\v(\mu_\varphi) e^F \omega_0^{n} \leq  L a \delta^n b^{-1}e^F \sqrt{F^2+1}\, \omega_0^{n} + L e^{\frac{b}{a\delta^n}}\omega_0^{n} \\
=& La \delta^n \omega_\psi^{n} + Le^{\frac{b}{a\delta^n}}\omega_0^{n} \\
\leq& L a\omega_{\delta \psi}^{n} + L e^{\frac{b}{a\delta^n}}\omega_0^{n}.        
    \end{split}
\end{equation*}
We are going to apply Theorem \ref{thm: uniform estimate} with $u=\varphi$, $\chi=\delta \psi$ and $f=Le^{b/(a\delta^n)}$. In fact, we have that $e^{b/(a\delta^n)} \in L^p$, for any $p\geq 1$ and, since $\int_X \omega_\varphi^{[n]}=1$, \cite[Theorem 1.3]{DDL2} implies that $P[\varphi]=0\geq \delta \psi$. In particular $P[\varphi]$ is less singular than $\delta \psi$. Moreover, the assumption in \eqref{eq: volume cap domination} is satisfied. Indeed, for any Borel set $E\subset X$
$$\int_E f\omega_0^{n}= Le^{\frac{b}{a\delta^n}} \vol_{\omega_0}(E) \leq L e^{\frac{b}{a\delta^n}} \exp\left( \frac{-C_4}{\capa_{\omega_0}(E)^{1/n}}\right) \leq L e^{\frac{b}{a\delta^n}} C_4 \capa_{\omega_0}(E)^{2}.$$
The inequality
$$ \vol_{\omega_0}(E) \leq \exp\left( \frac{-C_4}{\capa_{\omega_0}(E)^{1/n}}\right) $$where $C_4>0$ depends on $n$ and $\omega_0$, follows from \cite[eq. (12.1.3) and Lemma 12.2]{GZ17} (see also \cite[Proposition 2.10]{DDL1}). Using that
$\capa_{\omega_0}\leq (1-\delta)^{-n}\capa_{(1-\delta)\omega} $ and that $\capa_{(1-\delta)\omega_0} \leq \capa_{\delta \psi}$ (see \cite[Lemma 2.7]{DL14}) we get that
$$ \int_E f\omega_0^{n} \leq C_4 (1-\delta)^{-2n} \capa_{\delta \psi}(E)^2.$$
We can then infer that $\varphi \geq \delta \psi - C_4((1-a)^{-1}, e^{b/a\delta^n}, (1-\delta)^{-2n})$. Choosing $\delta$ small enough so that $\varepsilon-A \delta \geq0$ we obtain the claim with $C_3= A\,C_4$. 

It then follows that for any $x\in X$
$$H(x)\leq H(x_0)\leq C_0+C_3,$$
which concludes the proof.
\end{proof}

\begin{corollary}\label{coro-covid-19}
The functions $\psi, \varphi, F$ are uniformly bounded by a constant that only depends on $n,\omega_0$, $\v$, $\w$ and $\Ent(\varphi)$.
\end{corollary}
\begin{proof}
From Theorem \ref{step 1 C0 weighted} we know that $F\leq C-\varepsilon\psi+A\varphi \leq C-\varepsilon \psi$,
since $\sup_X \varphi=0$. Therefore
$$\int_X e^{2F} \omega_0^{n} \leq \tilde{C} \int_X e^{-2\varepsilon\psi} \omega_0^{n}.$$ 

Choosing $\varepsilon < \nu_{\omega_0}^{-1}$, by Theorem \ref{thm: Skoda} we get a uniform bound for $\|e^F\|_{L^2}$. It follows from Ko{\l}odziej uniform estimates (Theorem \ref{thm: Kol}), applied to the equation $\omega_\varphi^{n}=\frac{1}{\vv} e^F\omega_0^{n}$, that $\varphi \geq -C(\v,\|e^F\|_{L^2},\omega_0)$. In particular, since $\sup_X\varphi=0$ we do get a uniform control on $\|\varphi\|_{L^{\infty}}.$ Also,
$$ \int_X e^{2F} (F^2+1) \,\omega_0^{n} \leq  \int_X e^{4F}  \,\omega_0^{n}  \leq C' \int_X e^{-4\varepsilon\psi} \omega_0^{n}.$$ Once again, thanks to Theorem \ref{thm: Skoda}, choosing $\varepsilon \leq (2\nu_{\omega_0})^{-1}$ we get a uniform bound for $\|e^F\sqrt{F^2+1}\|_{L^2}$. Theorem \ref{thm: Kol} then gives a uniform control for $\|\psi\|_{L^{\infty}}.$ \\
We can then conclude from Theorem \ref{step 1 C0 weighted} together with the arguments above that $$F\leq C-\varepsilon\psi +A\varphi\leq -\varepsilon\inf_X \psi \leq C_4$$ for some uniform positive constant $C_4$. \\
It remains to prove a uniform lower bound for $F$. For this purpose we apply the minimum principle to $F+t\varphi$, with $t=2A_0+1$ where we recall that $A_0>0$ is such that $\Ric(\omega_0 )\leq A_0\omega_0 $. Using Definition \eqref{weighted trace}, Lemma \ref{lemma weighted Laplacian} and \eqref{bound-trace-weight}, we find

\begin{equation*}
    \begin{split}
        \Delta_{ \varphi, \v} (F + t\varphi) =& - \frac{\w(\mu_\varphi)}{\v(\mu_\varphi)} + 2 \Lambda_{\varphi, \v}(\Ric(\omega_0)) + t\Delta_\varphi \varphi + \frac{t}{\v(\mu_\varphi)}\sum_{a=1}^r \v_{,a}(\mu_\varphi)d^c\varphi(\xi_a)  \\   
       \leq & - \frac{\w(\mu_\varphi)}{\v(\mu_\varphi)} + 2   \Lambda_{\varphi}(\Ric(\omega_0)) + \tilde{C} + t ( n -  \Lambda_\varphi(\omega_0)) 
       + \frac{t}{\v(\mu_\varphi)} \sum_{a=1}^r \v_{,a}(\mu_\varphi)(\mu_\varphi^a-\mu_0^a) \\
       \leq & \frac{\nu}{\eta} + t n +(2A_0- t)\Lambda_\varphi (\omega_0) + \tilde{C}
       + \frac{t}{\v(\mu_\varphi)} \sum_{a=1}^r \v_{,a}(\mu_\varphi)(\mu_\varphi^a-\mu_0^a) \\
       \leq &  C -  \Lambda_\varphi (\omega_0)
               \\
           \leq  &  C -  \eta^{\frac{1}{n}}n e^{-F/n}, 
    \end{split}
\end{equation*}

\noindent where $C = \sup_X\frac{1}{\v(\mu_\varphi)}\left(  t \sum_{a=1}^r \v_{,a}(\mu_\varphi)(\mu_\varphi^a-\mu_0^a)\right)+\nu + t  n + \tilde{C}$. It is worth to note again that the constant $C$ is indeed independent of $\varphi$ since the image of $\mu_\varphi$ is the moment polytope $P$ for any $\T$-invariant K\"ahler potential in $\varphi \in \mathcal{K}(X,\omega_0)^\T$. \\
The last inequality follows from the fact that, given two positive $(1,1)$-forms $\alpha, \beta$ we have $\Lambda_\beta(\alpha) \geq n \big(\frac{\alpha^n}{\beta^n}\big)^{1/n}$. In particular 

\begin{equation}{\label{bound:trace}}
\Lambda_\varphi(\omega_0) \geq n \v(\mu_\varphi)^{1/n} e^{-F/n}.
\end{equation}
Now, let $x_0$ be a minimum point of the function $F+t\varphi$, then $0\leq C-\eta^{\frac{1}{n}}\, n e^{-F(x_0)/n}$, or equivalently $F(x_0) \geq -n \log(C \eta^{-\frac{1}{n}}/ n)$. For any $x\in X$, $F(x)+t\varphi(x) \geq F(x_0)+t\varphi(x_0)$, hence $F\geq -n \log(C\eta^{-\frac{1}{n}}/ n) -t \|\varphi\|_{L^{\infty}}.$
\end{proof}

\section{Integral $C^2$-estimates}{\label{s:integral:c2:esti}}
The goal is this section is to show the theorem below.

\begin{theorem}\label{thm:integral}
Assume $ \v$ is log-concave. Let $\varphi$ be a solution of \eqref{weighted system}. Then for any $p\geq 1$; there exists a constant $C>0$, depending on $p$, n, $\v$, $\|\varphi\|_{C^0}$, $\|F\|_{C^0}$, an upper bound on the Ricci form of $\omega_0$ and a lower bound of the holomorphic bisectional curvature of $\omega_0$ so that

\begin{equation*}
    \| \Lambda_0 (\omega_\varphi)\|_{L^p} \leq C.
\end{equation*}
\end{theorem}

We start with some lemmas true for any weight $\v>0$ which will be useful in the following.

\begin{lemma}{\label{l:trace}}

Let $\varphi$ be solution of $\eqref{weighted system}$. Then 

\begin{equation}{\label{ineq:trace:0phi}}
    \Lambda_0(\omega_\varphi) \leq C \Lambda_\varphi (\omega_0)^{n-1}, \qquad   \Lambda_\varphi(\omega_0) \leq C \Lambda_0(\omega_\varphi)^{n-1},
\end{equation}

\noindent where $C$ is a positive constant depending on $\|F\|_{C^0}$, $\v$ and $n$. Moreover, 

\begin{equation}\label{ineq:trace1}
    \Lambda_0(\omega_\varphi) C_1 \geq 1  \hspace*{2cm}    \Lambda_\varphi(\omega_0) C_2 \geq 1,
\end{equation}

\noindent for positive constants $C_1$, $C_2$ depending on $\|F\|_{C^0}$, $\v$, and $n$.
\end{lemma}

\begin{proof}
Recall that for any smooth $(1,1)$-form $\alpha$, $\beta$

\begin{equation*}
  n\left(\frac{\alpha^n}{\beta^n}\right)^{1/n} \leq  \Lambda_\beta(\alpha) \leq n \frac{\alpha^n}{\beta^n}\Lambda_\alpha(\beta)^{n-1}.
\end{equation*}
By \eqref{weighted system} and the second inequality above, we infer 
$$\Lambda_0(\omega_\varphi) \leq n \frac{e^F}{\v(\mu_\varphi)}  \Lambda_\varphi (\omega_0)^{n-1}. $$
Similarly $\Lambda_\varphi(\omega_0) \leq n \v(\mu_\varphi) {e^{-F}}  \Lambda_0 (\omega_\varphi)^{n-1}. $\\
On the other side, $ \Lambda_0(\omega_\varphi) \geq n \left( \frac{e^F}{\v(\mu_\varphi) } \right)^{1/n}$ and $ \Lambda_\varphi(\omega_0) \geq n \left( {\v(\mu_\varphi)  e^{-F}} \right)^{1/n}$.
\end{proof}

The following lemma is well-known; however, as it will be used very often, we give a statement below without proof.

\begin{lemma}\label{lemma_id grad} 
Let $S$  be a covariant $2$-tensor $S$ and $f_1, f_2$ two smooth functions on $X$. Then
\begin{eqnarray*}
 g_\varphi\big( df_1,g_\varphi(S,df_2)\big) &=& S(\nabla^\varphi  f_1, \nabla^\varphi f_2)
 \end{eqnarray*}
 and 
\begin{equation}\label{eq.grad 2}
 g_\varphi\big(S,  df_1 \otimes df_2\big)= S (\nabla^\varphi f_1, \nabla^\varphi f_2).
  \end{equation}
 In particular, when $S= \nabla^\varphi d h$, for some smooth function $h:X\rightarrow \R$ we have
 \begin{eqnarray*}
 g_\varphi\big( df_1,g_\varphi(\nabla^\varphi  dh,df_2)\big) &=& \nabla^\varphi dh (\nabla^\varphi  f_1, \nabla^\varphi f_2).
 \end{eqnarray*}
\end{lemma}
Let us stress that $ g_\varphi(S,df_2)$ has to be understood as the $1$-form: $ V \rightarrow g_\varphi(S(V,\cdot),df_2)$. 
    

\begin{lemma}{\label{l:diff:trace}}
For any smooth $(1,1)$-form $\beta$ we have the following identities between $1$-forms
  \begin{equation}{\label{diff:of:trace1}} 
   d\Lambda_\varphi(\beta)= \Lambda_\varphi(\nabla^0_{\cdot} \beta) - g_\varphi( \beta, \nabla^0_{\cdot} \omega_\varphi),\qquad d^c\Lambda_\varphi(\beta)= -\Lambda_\varphi(\nabla^0_{J\cdot} \beta) + g_\varphi( \beta, \nabla^0_{J\cdot} \omega_\varphi),\
   \end{equation}

\noindent where $g_\varphi( \beta, \nabla^0_\cdot \omega_\varphi)$ is the $1$-form defined as $V\rightarrow g_\varphi( \beta, \nabla^0_V \omega_\varphi)$ and $\Lambda_\varphi(\nabla^0_\cdot \beta)$ is the $1$-form defined as $V\rightarrow \Lambda_\varphi(\nabla^0_V \beta)$.

\end{lemma}

\begin{proof}
By definition of the trace

\begin{equation}{\label{tracbeta}}
    \beta \wedge \omega_\varphi^{[n-1]}= \Lambda_\varphi(\beta) \omega_\varphi^{[n]}.
\end{equation}    
Let $V$ be a vector field. Differentiating the LHS of the above equality by $\nabla^0_V$ we get

\begin{equation*}
  \begin{split}
    \nabla^0_V\left(\beta \wedge \omega_\varphi^{[n-1]}\right)
      =& \nabla^0_V \beta  \wedge \omega_\varphi^{[n-1]} + \beta \wedge \nabla^0_V (\omega_\varphi) \wedge \omega_\varphi^{[n-2]} \\
      =&   \Lambda_\varphi(\nabla^0_V \beta) \omega_\varphi^{[n]} +  \beta \wedge \nabla^0_V (\omega_\varphi) \wedge \omega_\varphi^{[n-2]}.
  \end{split}  
\end{equation*}

Differentiating the RHS of \eqref{tracbeta} by $\nabla^0_V$ and re-arranging the terms, we get

\begin{equation*}
    \begin{split}
\nabla^0_V\left(  \Lambda_\varphi \beta\right)\,  \omega_\varphi^{[n]}=&  \Lambda_\varphi(\nabla^0_V \beta) \omega_\varphi^{[n]} + \beta \wedge \nabla^0_V (\omega_\varphi) \wedge \omega_\varphi^{[n-2]} -   \Lambda_\varphi(\beta) \nabla^0_V(\omega_\varphi)\wedge \omega_\varphi^{[n-1]} \\
=&\Lambda_\varphi\left(\nabla^0_V\beta\right) \omega_\varphi^{[n]} + \beta \wedge \nabla^0_V (\omega_\varphi) \wedge \omega_\varphi^{[n-2]} -   \Lambda_\varphi(\beta) \Lambda_\varphi(\nabla^0_V(\omega_\varphi))\omega_\varphi^{[n]} \\
=&\Lambda_\varphi\left(\nabla^0_V\beta\right) \omega_\varphi^{[n]} - g_\varphi( \beta, \nabla^0_V \omega_\varphi)\omega_\varphi^{[n]},
    \end{split}
\end{equation*}
where to pass to the last line we use the well know identity $ \Lambda_\varphi(\beta) \Lambda_\varphi(\alpha)\omega_\varphi^{[n]}=  g_\varphi( \beta, \alpha)\omega_\varphi^{[n]} + \beta\wedge \alpha \wedge \omega_\varphi^{[n-2]}$  (see e.g. \cite[eq. (1.12.5)]{Gau}), for any $(1,1)$-forms $\alpha, \beta$. \\
This conclude the proof of the first identity, since $d  \Lambda_\varphi(\beta)(V)=\nabla^0_V\left(\Lambda_\varphi(\beta)\right)$. The second one simply follows from the first one and the  that $d^cf =-df(J\cdot) $ for any smooth function $f$.
\end{proof}

The following lemma is a generalization of Yau inequality \cite[(2.10)]{Yau78} for $F$ solution of the weighted Monge-Ampère equation \eqref{weighted system} and for the weighted Laplacian.

\begin{lemma}{\label{l:yau}}
The following holds true:
\begin{equation}{\label{yau:estim3}}
\begin{split}
    \Delta_{\varphi, \v} \Lambda_{0}(\omega_\varphi) \geq& \Delta_{0} F  - \langle \hess( \log \v )(\mu_\varphi), g_0(d\mu_\varphi,d\mu_\varphi)\rangle
     \\
   &  +  |\nabla^0\omega_\varphi|^2_{g_0\otimes g_\varphi}- C \Lambda_0(\omega_\varphi)\Lambda_\varphi(\omega_0) -  C,
\end{split}    
\end{equation}
where $C$ depends on $\omega_0$, $\|F\|_{C^0}$, $\v$  and a lower bound for the holomorphic bisectional curvature of $\omega_0$. In particular, 
\begin{equation}{\label{yau:estim2}}
    \Delta_{\varphi,\v} \Lambda_{0}(\omega_\varphi) \geq \Delta_{0} F +  |\nabla^0\omega_\varphi|^2_{g_0\otimes g_\varphi} - C_1 \Lambda_0(\omega_\varphi)^{n}- C_2.
\end{equation}

\noindent Moreover, if $\v$ is log-concave, then 

\begin{equation}{\label{yau:estim}}
    \Delta_{\varphi, \v} \Lambda_{0}(\omega_\varphi) \geq \Delta_{0} F  
    +  |\nabla^0\omega_\varphi|^2_{g_0\otimes g_\varphi}- C \Lambda_0(\omega_\varphi)\Lambda_\varphi(\omega_0) -  C,   
\end{equation}

\end{lemma}
In the above Lemma, 
$$ |\nabla^0\omega_\varphi|^2_{g_0\otimes g_\varphi}= \sum_{k=1}^{n}|\nabla^0_{e_k} \omega_\varphi|^2_{g_\varphi}+\sum_{k=1}^{n}|\nabla^0_{Je_k} \omega_\varphi|^2_{g_\varphi}$$
where $(e_k,Je_k)_{k=1}^{n}$ in a orthonormal frame with respect to $g_0$.



\begin{proof}
Following Yau's computations \cite[eq. (2.7) and (2.9)]{Yau78}, the Laplacian with respect to $\omega_\varphi$ of the trace $\Lambda_0(\omega_\varphi)$ is given by 
\begin{align}\label{eq:Delta-tr0}
 \begin{split}   \Delta_{\varphi}(\Lambda_0 (\omega_\varphi))
    =& \Delta_0\log\left(\frac{\omega_\varphi^n}{\omega_0^n}\right) + |\nabla^0\omega_\varphi|_{g_0\otimes g_\varphi}^2\\
    &+ \sum_{i, j=1}^n \Big(R^0(e_i,e_j,e_j,e_i)+R^0(Je_i,e_j,e_j,Je_i)\Big)\frac{|e_j|_\varphi^2-|e_i|_\varphi^2}{|e_i|_\varphi^2}\\
    =&\Delta_0 F-\Delta_0 \log \v(\mu_\varphi) + |\nabla^0\omega_\varphi|_{g_0\otimes g_\varphi}^2\\
   & + \sum_{i, j=1}^n \Big(R^0(e_i,e_j,e_j,e_i)+R^0(Je_i,e_j,e_j,Je_i)\Big)\frac{|e_j|_\varphi^2-|e_i|_\varphi^2}{|e_i|_\varphi^2}
    \end{split}
\end{align}

where $R^0$ is the curvature tensor of the reference metric $g_0$. The $\v$-weighted Laplacian of $\Lambda_0 (\omega_\varphi)$ is given by
\begin{equation}
\begin{split}\label{eq:Delta_v_0}
    \Delta_{\varphi,\v}(\Lambda_0 (\omega_\varphi))=&  \Delta_\varphi(\Lambda_0 (\omega_\varphi))+ g_\varphi(d(\log\v(\mu_\varphi)),d\Lambda_0(\omega_\varphi))\\
    =&\Delta_\varphi(\Lambda_0 (\omega_\varphi))+g_\varphi \left(d(\log\v(\mu_\varphi)),\Lambda_0(\nabla^0\omega_\varphi \right))
\end{split}
\end{equation}
where the second equality follows from Lemma \ref{l:diff:trace}: 
\[
d\Lambda_0(\omega_\varphi)=\Lambda_0(\nabla^0\omega_\varphi)-g_0(\omega_\varphi,\nabla^0\omega_0)=\Lambda_0(\nabla^0\omega_\varphi).
\]
In a basis $(\xi_a)_{a=1,\cdots, r}$ of $\tor$ the second term in \eqref{eq:Delta_v_0} is given by:
\[
\begin{split}
    g_\varphi \left(d(\log\v(\mu_\varphi)),\Lambda_0(\nabla^0\omega_\varphi) \right)=&\sum_{a=1}^r(\log\v)_{,a}(\mu_\varphi)g_\varphi(d\mu_\varphi^a,\Lambda_0(\nabla^0\omega_\varphi))\\
    =&-\sum_{a=1}^r(\log\v)_{,a}(\mu_\varphi) \Lambda_0(\nabla^0_{J\xi_a}\omega_\varphi)
\end{split}
\]
For arbitrary vector fields $U,V,W$ we compute
\[
\begin{split}
    (\nabla^0_U\omega_\varphi)(V,W)=&\mathcal{L}_U(\omega_\varphi(V,W))-\omega_\varphi(\nabla^0_UV,W)-\omega_\varphi(V,\nabla^0_UW)\\
    =&(\mathcal{L}_U\omega_\varphi)(V,W)+\omega_\varphi([U,V],W)+\omega_\varphi(V,[U,W])\\
    &-\omega_\varphi(\nabla^0_UV,W)-\omega_\varphi(V,\nabla^0_UW)\\
    =&(\mathcal{L}_U\omega_\varphi)(V,W) -\omega_\varphi(\nabla^0_VU,W)-\omega_\varphi(V,\nabla^0_WU)\\
    =&(\mathcal{L}_U\omega_\varphi)(V,W) -[\omega_\varphi(\nabla^0_VU,W)-\omega_\varphi(\nabla^0_WU,V)]\\
    =&\left(\mathcal{L}_U\omega_\varphi-2(\nabla^0 U\lrcorner\omega_\varphi)^{\rm skw}\right)(V,W)
\end{split}
\]
where $(\nabla^0 U\lrcorner\omega_\varphi)^{\rm skw}$ is the skew-symmetric part of the bilinear form $(V,W)\mapsto \omega_\varphi(\nabla^0_V U,W)$. Plugging $U=J\xi_a$ in the above equation, using $\mathcal{L}_{J\xi_a}\omega =dd^c \mu_\varphi^a$ and taking the trace relative to $\omega_0$ we obtain
\[
\begin{split}
\Lambda_0 \nabla^0_{J\xi_a}\omega_\varphi=&-\Delta_0\mu_\varphi^a+2\Lambda_0\left( g_\varphi ( \nabla^0\xi_a, \cdot) \right)^{\rm skw}.
\end{split}
\]
Substituting back we get
\begin{equation*}
\begin{split}\label{eq:g_nabla}
  & g_\varphi \left(d(\log\v(\mu_\varphi)),\Lambda_0\nabla^0\omega_\varphi \right)\\
    &= -\sum_{a=1}^r (\log\v)_{,a}(\mu_\varphi) \big(-\Delta_0\mu_\varphi^a+2\Lambda_0\left( g_\varphi ( \nabla^0\xi_a, \cdot) \right)^{\rm skw}\big)\\
    &= \frac{1}{\v(\mu_\varphi)} \sum_{i=1}^r \v_{,a}(\mu_\varphi)  \Delta_0 \mu_\varphi^a  -2\sum_{a=1}^r (\log\v)_{,a}(\mu_\varphi) \Lambda_0\left( g_\varphi ( \nabla^0\xi_a, \cdot) \right)^{\rm skw}\\
&= \frac{1}{\v(\mu_\varphi)} \Delta_0 \v(\mu_\varphi) -   \frac{1}{\v(\mu_\varphi)}\sum_{a,b=1}^r \v_{,ab}(\mu_\varphi) g_0(d\mu^a_\varphi, d \mu^b_\varphi)\\
&\quad -2\sum_{a=1}^r (\log\v)_{,a}(\mu_\varphi) \Lambda_0\left( g_\varphi ( \nabla^0\xi_a, \cdot) \right)^{\rm skw},
\end{split}
\end{equation*}

where we use \eqref{eq: lap v} for the last equality. We continue the computation

\begin{equation}
\begin{split}
g_\varphi \left(d(\log\v(\mu_\varphi)),\Lambda_0\nabla^0\omega_\varphi \right) &=  \Delta_0 \log\v(\mu_\varphi)   + \frac{1}{\v(\mu_\varphi)^2} \sum_{a,b=1}^r \v_{,a}(\mu_\varphi)\v_{,b}(\mu_\varphi) g_0(d\mu_\varphi^a, d\mu_\varphi^b) \\
&\quad -   \frac{1}{\v(\mu_\varphi)}  \sum_{a,b=1}^r\v_{,ab}(\mu_\varphi)  g_0(d\mu^a_\varphi, d \mu^b_\varphi) \\
&\quad -2\sum_{a=1}^r(\log\v)_{,a}(\mu_\varphi) \Lambda_0\left( g_\varphi ( \nabla^0\xi_a, \cdot) \right)^{\rm skw}\\
&=\Delta_0 \log\v(\mu_\varphi) - \langle \hess( \log\circ \v)(\mu_\varphi), g_0(d\mu_\varphi,d\mu_\varphi)\rangle\\
&\quad -2\sum_{a=1}^r(\log\v)_{,a}(\mu_\varphi) \Lambda_0\left( g_\varphi ( \nabla^0\xi_a, \cdot) \right)^{\rm skw}
\end{split}
\end{equation}
Substituting \eqref{eq:g_nabla} and \eqref{eq:Delta_v_0} back into \eqref{eq:Delta-tr0} we obtain
\[
\begin{split}
    \Delta_{\varphi,\v}(\Lambda_0 (\omega_\varphi))
    =& \Delta_0 F- \langle \hess( \log\circ \v)(\mu_\varphi), g_0(d\mu_\varphi,d\mu_\varphi)\rangle\\
    &-2\sum_{a=1}^r(\log \v)_{,a}(\mu_\varphi) \Lambda_0\left( g_\varphi ( \nabla^0\xi_a, \cdot) \right)^{\rm skw} + |\nabla^0\omega_\varphi|^2_{g_0\otimes g_\varphi}\\
    &+ \sum_{i, j=1}^n \Big(R^0(e_i,e_j,e_j,e_i)+R^0(Je_i,e_j,e_j,Je_i)\Big)\frac{|e_j|_\varphi^2-|e_i|_\varphi^2}{|e_i|_\varphi^2}
\end{split}
\] 

We now want a lower bound for $ \Delta_{\varphi,\v}(\Lambda_0 (\omega_\varphi))$.  We do have
\[
\begin{split}
 &\Big(R^0(e_i,e_j,e_j,e_i)+R^0(Je_i,e_j,e_j,Je_i)\Big)\frac{|e_j|_\varphi^2-|e_i|_\varphi^2}{|e_i|_\varphi^2} \\
& \geq     \Big(R^0(e_i,e_j,e_j,e_i)+R^0(Je_i,e_j,e_j,Je_i)\Big)\frac{|e_j|_\varphi^2}{|e_i|_\varphi^2}   -C, \\
    & \geq- C\sum_{i\neq j}^n  \frac{|e_j|_\varphi^2}{|e_i|_\varphi^2}-C \\
   & \geq  - C\sum_{i=1}^n  \frac{1}{|e_i|_\varphi^2}\sum_{j=1}^n |e_j|_\varphi^2 -C\\
    &=  -C \Lambda_\varphi (\omega_0) \Lambda_0 (\omega_\varphi)  -C
       \end{split}
       \]
       
  Moreover, since by definition $\Lambda_0(\theta)= g_0(\omega_0, \theta)$ for any $(1,1)$-form $\theta$, it follows from Cauchy-Schwartz that for any $a=1, \cdots, r$
  \[
  \begin{split}
    |\Lambda_0\left( g_\varphi ( \nabla^0\xi_a, \cdot) \right)^{\rm skw}|&=\left|g_0\left(\omega_0,\left( g_\varphi ( \nabla^0\xi_a, \cdot) \right)^{\rm skw}\right)\right|\\
    &\leq |\omega_0|_{g_0}\left|\left( g_\varphi ( \nabla^0\xi_a, \cdot) \right)^{\rm skw}\right|_{g_0}\\
    &=|\omega_0|_{g_0}\left|\left( \omega_\varphi ( J\nabla^0\xi_a, \cdot) \right)^{\rm skw}\right|_{g_0}\\
    &\leq C |\omega_0|_{g_0}|\omega_\varphi|_{g_0}\\
    &\leq C\Lambda_0(\omega_\varphi),
\end{split}
\]
where in the forth inequality the constant $C>0$ is such that $|J\nabla^0 \xi_i |_{g_0}\leq C$.
Hence,
     \begin{eqnarray*}
  -2\sum_{a=1}^r(\log \v)_{,a}(\mu_\varphi) \Lambda_0\left( g_\varphi ( \nabla^0\xi_a, \cdot) \right)^{\rm skw}& \geq &   -C  \Lambda_0(\omega_\varphi) \geq -C \Lambda_0(\omega_\varphi)\Lambda_\varphi(\omega_0), 
    \end{eqnarray*}
for $C$ positive constant independent of $\varphi$. We also use Lemma \ref{l:trace} for the last inequality. \\
We have then showed \eqref{yau:estim3} and \eqref{yau:estim}. For \eqref{yau:estim2}, we observe that 

\begin{equation}\label{bound_hess}
    \begin{split}
  - \langle \hess( \log\circ \v)(\mu_\varphi), g_0(d\mu_\varphi,d\mu_\varphi)\rangle \geq& -C\sum_{a,b=1}^r g_0(d\mu^a_\varphi,d\mu^b_\varphi) \\
  \geq &-C\sum_{a,b=1}^r |d\mu_\varphi^a|_0|d\mu_\varphi^b|_0 \\
  =&-C\sum_{a,b=1}^r |\omega_\varphi(\xi^a, \cdot )|_0|\omega_\varphi(\xi^b, \cdot )|_0 \\
  \geq& - C_1 \Lambda_0(\omega_\varphi)^2 \\
  \geq &- C_2\Lambda_0(\omega_\varphi)^n,
    \end{split}
\end{equation}
where for the last inequality we use Lemma \ref{l:trace} and that $n \geq2$.
\end{proof}

We now present a generalization of \cite[Lemma 2.2]{CGP} to the weighted setting:

\begin{lemma}\label{lemma: CGP}
    The following inequality holds true:
  \begin{equation*}
        \Delta_{\varphi,\v}  \log \Lambda_0(\omega_\varphi )\geq \frac{1}{\Lambda_0(\omega_\varphi)} \Delta_{0}F - B \Lambda_\varphi(\omega_0) -\frac{1}{\Lambda_0(\omega_\varphi)}\langle \hess( \log\circ \v)(\mu_\varphi), g_0(d\mu_\varphi,d\mu_\varphi)\rangle,
    \end{equation*}
In particular, using \eqref{bound_hess} we get $$\Delta_{\varphi,\v}  \log \Lambda_0(\omega_\varphi )\geq \frac{1}{\Lambda_0(\omega_\varphi)} \Delta_{0}F - B \Lambda_\varphi(\omega_0) -{C}{\Lambda_0(\omega_\varphi)},$$
   where $C>0$ depends only on $\v$, and $B>0$ depends on $\omega_0, \v,n, \|F\|_{C^0}$ and a lower bound for the holomorphic bisectional curvature of $\omega_0$. Moreover, if $\v$ is log-concave we get
  \begin{equation*}
        \Delta_{\varphi,\v}  \log \Lambda_0(\omega_\varphi )\geq \frac{1}{\Lambda_0(\omega_\varphi)} \Delta_{0}F - B \Lambda_\varphi(\omega_0).
    \end{equation*}
   
\end{lemma}

\begin{proof}
    A direct computation (using the explicit expression of the weighted Laplacian given in Lemma \ref{lemma weighted Laplacian}) shows that

\begin{equation*}
    \Delta_{\varphi,\v} \log \Lambda_0(\omega_\varphi) = \frac{1}{\Lambda_0(\omega_\varphi)}  \Delta_{\varphi,\v} \Lambda_0(\omega_\varphi) - \frac{1}{\Lambda_0(\omega_\varphi)^2} | d \Lambda_0(\omega_\varphi)|_\varphi^2. 
\end{equation*}
We focus on the second term of the RHS of the above equality. We let $(e_k, Je_k)_{k=1}^{n}$ be a orthonormal frame with respect to $\omega_0$. By Lemma \ref{l:diff:trace} and the fact that $g_0(\omega_\varphi, \nabla^0 \omega_0)=0$ it follows that
\begin{equation*}
| d \Lambda_0(\omega_\varphi)|_\varphi^2= | \nabla^0 \Lambda_0(\omega_\varphi)|_\varphi^2 =|\Lambda_0(\nabla^0\omega_\varphi)|_\varphi^2, 
   \end{equation*}

By \cite[section (3.2) p. 99]{Siu} we have that
\begin{equation*}
\begin{split}
   |\Lambda_0(\nabla^0\omega_\varphi)|_\varphi^2=&|\nabla^0 \Lambda_0(\omega_\varphi)|_\varphi^2 \\
   \leq&\Lambda_0(\omega_\varphi) \left( \sum_{k=1}^{n} | \nabla^0_{e_k} \omega_\varphi |_\varphi^2 + \sum_{k=1}^{n} | \nabla^0_{Je_k} \omega_\varphi |_\varphi^2\right)\\
   =&\Lambda_0(\omega_\varphi) |\nabla^0\omega_\varphi|^2_{g_0\otimes g_\varphi}.
\end{split}  
\end{equation*}
From \eqref{yau:estim3} we have
\begin{equation*}
  \Delta_{\varphi,\v} \Lambda_0(\omega_\varphi) \geq \Delta_{0} F  - \langle \hess( \log \v )(\mu_\varphi), g_0(d\mu_\varphi,d\mu_\varphi)\rangle+    |\nabla^0\omega_\varphi|^2_{g_0\otimes g_\varphi}- C_1 \Lambda_0(\omega_\varphi)\Lambda_\varphi(\omega_0) - C_2 .
\end{equation*}  
 We then get 
\begin{equation*}
\begin{split}
    \Delta_{\varphi,\v} \log \Lambda_0(\omega_\varphi) \geq& \frac{1}{\Lambda_0(\omega_\varphi)} \left( \Delta_{0} F +  |\nabla^0\omega_\varphi|^2_{g_0\otimes g_\varphi} - C_1 \Lambda_0(\omega_\varphi)\Lambda_\varphi(\omega_0)  -C_2  \right) \\
    & -\frac{1}{\Lambda_0(\omega_\varphi)}\langle \hess( \log\circ \v)(\mu_\varphi), g_0(d\mu_\varphi,d\mu_\varphi)\rangle -  \frac{1}{\Lambda_0(\omega_\varphi)} |\nabla^0\omega_\varphi|^2_{g_0\otimes g_\varphi}\\
    \geq& \frac{1}{\Lambda_0(\omega_\varphi)}  \Delta_{0} F - C_1\Lambda_\varphi(\omega_0)- \frac{C_2}{\Lambda_0(\omega_\varphi)} \\
    &-\frac{1}{\Lambda_0(\omega_\varphi)}\langle \hess( \log\circ \v)(\mu_\varphi), g_0(d\mu_\varphi,d\mu_\varphi)\rangle \\
     \geq &\frac{1}{\Lambda_0(\omega_\varphi)}  \Delta_{0} F - (C_1+C_2)\Lambda_\varphi(\omega_0)\\
     &-\frac{1}{\Lambda_0(\omega_\varphi)}\langle \hess( \log\circ \v)(\mu_\varphi), g_0(d\mu_\varphi,d\mu_\varphi)\rangle,
\end{split}    
\end{equation*} 
where in the last inequality we used that $\Lambda_\varphi(\omega_0) \Lambda_0(\omega_\varphi)\geq 1$.
\end{proof}

We are now ready to give the proof of Theorem \ref{thm:integral}:

\begin{proof}[proof of Theorem \ref{thm:integral}]
Consider $$u:= e^{-\gamma( F+\lambda \varphi)} \Tr\, (\omega_\varphi),$$ 

where $\lambda>1$ and $\gamma>1$ are uniform constants to be chosen in a suitable way in the following. 

The $\v$-Laplacian of $u$ is

\begin{equation*}
\begin{split}
    \Delta_{\varphi,\v} u  =&    \Delta_{\varphi,\v} e^{\log u} \\
   =& \Delta_\varphi e^{\log u} + g_\varphi\left( d\log \v(\mu_\varphi), d\left(e^{\log u}\right)\right) \\
    =&  \Lambda_{\varphi}\left( \frac{e^{\log u}}{u^2} du\wedge d^cu\right) + e^{\log u} \Delta_{\varphi,\v} \log u \\
    = & \frac{e^{\log(u)}}{{u^2}}|du|^2_\varphi + e^{\log u} \Delta_{\varphi,\v} \log u \\
    \geq & e^{\log u} \Delta_{\varphi,\v} \log u \\
    =& -\gamma u\Delta_{\varphi,\v}(F + \lambda \varphi)   + u \Delta_{\varphi,\v} \left( \log \Lambda_0(\omega_\varphi)\right)  \\
\end{split}
\end{equation*}

From \eqref{weighted system}, the fact that $\Ric(\omega_0)\leq A_0 \omega_0$ and that $\Lambda_{\varphi, \v} (\omega_\varphi) \leq n+C$ (by \eqref{def:weighted:trace}) we deduce that
 
 \begin{equation*}
     \begin{split}
          \Delta_{\varphi,\v}(F + \lambda \varphi)=&  -\frac{\w(\mu_\varphi)}{\v(\mu_\varphi)} + 2 \Lambda_{\varphi, \v}(\Ric(\omega_0))
+ \lambda \Lambda_{\varphi,\v} dd^c\varphi \\
    =& -\frac{\w(\mu_\varphi)}{\v(\mu_\varphi)} + 2 \Lambda_{\varphi, \v}(\Ric(\omega_0))
    +\lambda\left( \Lambda_{\varphi,\v}(\omega_\varphi)- \Lambda_{\varphi, \v}(\omega_0)\right) \\
    \leq&  C_1 + (2A_0 - \lambda) \Lambda_\varphi(\omega_0) +\lambda n +C_2 \\
=&  C_3+ (2A_0 - \lambda) \Lambda_\varphi(\omega_0)
     \end{split}
 \end{equation*}


Thus, combining the above inequalities together with Lemma \ref{lemma: CGP} leads us to:

\begin{equation*}
\begin{split}
\Delta_{\varphi,\v} u  \geq & e^{-\gamma (F+\lambda \varphi)} \bigg(-\gamma C_3 \Tr( \omega_\varphi) +  \Delta_{0}F+(\lambda \gamma -2A_0 \gamma - B)  \Tr( \omega_\varphi) \Lambda_{\varphi}(\omega_0) \bigg) 
\end{split}
\end{equation*}
Using \eqref{weighted system}, (the proof of) Lemma \ref{l:trace} and the fact that $n \geq 2$, we have

$$\Tr(\omega_\varphi) \Lambda_\varphi (\omega_0) \geq n^{-\frac{1}{n-1}} e^{\frac{-F}{n-1}}\eta^{\frac{1}{n-1}} \Tr(\omega_\varphi)^{1+\frac{1}{n-1}}:= a(n, \eta) e^{\frac{-F}{n-1}} \Tr(\omega_\varphi)^{1+\frac{1}{n-1}},$$

\noindent where $\eta$ is defined in \eqref{bound:v}.\\
We now choose $\lambda \geq 4 \max (2A_0, B)$ (in order to have $\lambda \gamma -2A_0\gamma -B\geq \frac{\lambda \gamma }{2}$) so that
\begin{equation}\label{lower bound lap}
\Delta_{\varphi,\v} u  \geq  -\gamma C_3  u + \frac{\lambda \gamma}{2} a(n, \eta) e^{-\frac{F}{n-1}}  \Tr(\omega_\varphi)^{\frac{1}{n-1}} \,u +e^{-\gamma (F+\lambda \varphi)} \Delta_{0} F.
\end{equation}
Now, since $|du |^2_\varphi \Tr\, \omega_\varphi \geq |du|_0^2 $ holds pointwise, we write
$$
\frac{1}{2p+1} \Delta_{\varphi,\v} u^{2p+1} = u^{2p} \Delta_{\varphi,\v} u + 2p u^{2p-1} |du|^2_\varphi \\
\geq  u^{2p} \Delta_{\varphi,\v} u +2p u^{2p-2} e^{-\gamma (F+\lambda \varphi)} |du|_0^2.
$$

Thus, combining the above inequality with \eqref{lower bound lap}, and since $\Delta_{\varphi, \v}$ is self-adjoint w.r.t. to $\v(\mu_\varphi)\omega_\varphi^n$ (Lemma \ref{l:weightedlap:2}), we get

\begin{equation}\label{integral bound}
\begin{split}
 0 = \frac{1}{2p+1} \int_X \Delta_{\varphi,\v} u^{2p+1} \v(\mu_\varphi)\omega_\varphi^{[n]} 
&\geq   2p \int_X u^{2p-2}|d u |_0^2  e^{-\gamma (F+\lambda \varphi)+F} \omega_0^{[n]}-\gamma C_3  \int_X u^{2p+1} e^F \omega_0^{[n]}\\
&\quad+\frac{\gamma \lambda}{2} a(n, \eta) \int_X u^{2p+1} e^{\left(\frac{n-2}{n-1}\right)F}  \Tr\,(\omega_\varphi)^{\frac{1}{n-1}} \omega_0^{[n]} \\
&\quad+\int_X u^{2p} e^{-\gamma (F+\lambda \varphi)+F} \Delta_{0} F\, \omega_0^{[n]}.
\end{split}
\end{equation}
Next, we focus on finding a suitable lower bound for the last term involving the Laplacian of $F$. This estimate goes as in \cite{DD21} but we show it for completeness and for reader's convenience. 

 Set $G:=(1-\gamma)F-\gamma\lambda \varphi$. A formal trick gives that
\begin{equation*}
\begin{split}
I:=& -\int_X u^{2p} \Delta_{0} F\, e^{G} \, \omega_0^{[n]}\\
=& \frac{1}{\gamma-1} \int_X u^{2p} \,\Delta_0G  \, e^{G} \, \omega_0^{[n]}+\frac{\gamma \lambda }{\gamma-1} \int_X u^{2p} \Delta_{0}\varphi\, e^{G} \, \omega_0^{[n]} \\
:=&   I_1+I_2.
\end{split}
\end{equation*}
Integration by part gives

\begin{eqnarray}\label{bound I1}
\nonumber I_1 &=  &
-\frac{1}{\gamma-1} \int_X u^{2p} |d G|_0^2 e^{G} \, \omega_0^{[n]} - \frac{2p}{\gamma -1} \int_X u^{2p-1}  \,e^{G} \, du\wedge d^c G\wedge \omega_0^{[n-1]}\\
\nonumber &\leq &  -\frac{1}{2(\gamma-1)} \int_X u^{2p} |d G|_0^2 e^{G} \omega_0^{[n]} + \frac{2p^2}{\gamma -1} \int_X u^{2p-2} |d u|_0^2 e^G  \, \omega_0^{[n]}\\
&\leq & \frac{ 2p^2}{\gamma -1} \int_X u^{2p-2} |du|_0^2 e^G \,\omega_0^{[n]},
\end{eqnarray}

where in the first inequality we used the fact that 
$$\left|2pu^{2p-1} \frac{du\wedge d^c G\wedge \omega_0^{[n-1]}}{\omega_0^{[n]}}\right| \leq \frac{(2p)^2 }{2} u^{2p-2} | d u|_0^2 +\frac{1}{2} u^{2p}|d G|_0^2, $$ by Young's inequality. Also,




\begin{equation}\label{bound I2}
\begin{split}
I_2 =& \frac{\gamma \lambda }{\gamma-1} \int_X u^{2p} (\Lambda_0(\omega_\varphi)-n)\, e^{G} \, \omega_0^{[n]}\\
\leq&\frac{\gamma \lambda }{\gamma-1} \int_X u^{2p} \Lambda_0(\omega_\varphi) e^G\, \omega_0^{[n]} 
= \frac{\gamma \lambda }{\gamma-1} \int_X u^{2p+1}  e^F\, \omega_0^{[n]}.
\end{split}
\end{equation}
Suppose  $p>1$. Then Combining \eqref{integral bound}, \eqref{bound I1}, \eqref{bound I2} and choosing $\gamma$ big enough (say $\gamma=ap$ with $a>>0$)  we obtain
\begin{equation}
\begin{split}\label{big-variant}
0&\geq  2\left( p-\frac{ p^2}{\gamma -1}\right)   \int_X u^{2p-2} |d u|_0^2 e^{G} \,  \\
&\quad - \gamma \left(C_3+\frac{\lambda}{\gamma-1}\right) \int_X u^{2p+1} e^F \omega_0^{[n]} +\frac{\gamma \lambda}{2} a(n, \eta) \int_X u^{2p+1} e^{\left(\frac{n-2}{n-1}\right)F} \Tr(\omega_\varphi)^{\frac{1}{n-1}} \omega_0^{[n]}\\
&\geq - C_4 \int_X u^{2p+1} e^F \omega_0^{[n]} + C_5 \int_X u^{2p+1} e^{\left(\frac{n-2}{n-1}\right)F} \Tr(\omega_\varphi)^{\frac{1}{n-1}} \omega_0^{[n]}.
\end{split}
\end{equation}

\noindent where the constant $C_4>0$, $C_5 >0$ depends on $\|F\|_0$, $\| \varphi \|_0$. Observe that in \eqref{big-variant}, the choice
of $\gamma$ ensures that $p-\frac{ p^2}{\gamma-1} > 0$. Using H\"older inequality we can conclude that

\begin{equation*}
    \| \Lambda_0(\omega_\varphi)\|^{2p+1+\frac{1}{n-1}}_{L^{2p+1+\frac{1}{n-1}}} \leq C  \| \Lambda_0(\omega_\varphi)\|^{2p+1}_{L^{2p+1}} \leq  C' \| \Lambda_0(\omega_\varphi)\|^{2p+1}_{L^{2p+1+\frac{1}{n-1}}}
\end{equation*}

This gives the statement for $p> 3$, hence for $p\geq 1$ thanks to H\"older inequality.
\end{proof}

\section{$C^2$-estimates}{\label{s:c2:esti}}
 The goal to this section is to prove the theorem below.

\begin{theorem}{\label{t:c2:estim}}
Suppose $\v$ is log-concave. Let $\varphi$ be a solution of \eqref{weighted system}. Then there exists a positive constant $C$ depending on $\omega_0$, $\v$, $\w$, $\|F\|_{C^0}$ and $\| \varphi \|_{C^0}$ such that

\begin{equation*}
    \max_X\left(|dF|^2_\varphi + \Lambda_0(\omega_\varphi)\right) \leq C.
\end{equation*}
\end{theorem}

We start with several lemmas and propositions which are true for any weights $\v>0$. The assumption of the concavity of $\log \v$ is only needed in proof of Theorem \ref{t:c2:estim} below.

We begin by a technical lemma:

\begin{lemma}\label{ddc-grad}
For any smooth function $f$ on $X$ we have
$$ dd^c f(\cdot, J\cdot )= 2\nabla^{\varphi,+}df (\cdot, \cdot),$$
where $\nabla^{\varphi, +}\alpha$ for a 1-form $\alpha$ is defined by

\begin{equation*}
\nabla_V^{\varphi, \pm}\alpha (W):=\frac{1}{2}\left(\nabla^\varphi_V\alpha(W) {\pm}\nabla^\varphi_{JV}\alpha(JW)\right).
\end{equation*}
\end{lemma}
Let us stress that the above identity holds for the Levi-Civita connection $\nabla^{\varphi}$ associated to any K\"ahler metric $\omega_\varphi$. Observe as well that $ \nabla^{\varphi,+}$ and $ \nabla^{\varphi,-}$ are orthogonal w.r.t. $g_\varphi$.
\begin{proof}
Let $V,W$ be two vector fields, we compute
\[
\begin{split}
    dd^cf(V,W)= & \nabla^\varphi_V d^cf(W)-\nabla^\varphi_W d^cf(V) \\
    =&(\nabla^\varphi_V Jdf)(W)-(\nabla^\varphi_W Jdf)(V)\\
    =&(J\nabla^\varphi_V df)(W)-(J\nabla^\varphi_W df)(V)\\
    =&-(\nabla^\varphi_V df)(JW)+(\nabla^\varphi_W df)(JV),
\end{split}
\]
where we use that $\nabla^\varphi J=0$. It follows that 
\[
dd^cf(V,JW)=-\nabla^\varphi_V df(J^2W)+\nabla^\varphi_{JV} df(JW)=2\nabla_V^{\varphi,+}df(W).
\]
\end{proof}

The first main step to prove Theorem \ref{t:c2:estim} is this key proposition:

\begin{prop}{\label{p:weightedlap:u}}
Let $u:=e^{\frac{F}{2}} | dF |_\varphi^2 + K \Lambda_0(\omega_\varphi)$. Then there exists positive constants $K$ and
$C$ depending on $\omega_0$, n, $\v$, $\w$, $\|F\|_{C^0}$ and $\|\varphi\|_{C^0}$ and an upper bound for the Ricci curvature, such that the function $u$ satisfies the following differential
inequality

\begin{equation*}
    \Delta_{\varphi, \v}u \geq  - C \Lambda_0(\omega_\varphi)^{3n-3}u.
\end{equation*}
\end{prop}
As a remark, we note that $u$ is uniformly bounded. Indeed, by the proof of Lemma \ref{l:trace}
$$u \geq K \Lambda_0(\omega_\varphi) \geq K n \bigg(\frac{e^F}{\v(\mu_\varphi)}\bigg)^{1/n}  \geq \frac{n K}{L^{1/n}} e^\frac{-\|F\|_{C^0}}{n},$$
where we recall that $\v\leq L$.

\smallskip
Our first step to show Proposition \ref{p:weightedlap:u} is to bound from below uniformly the weighted Laplacian of $e^{F/2} |dF|_\varphi^{2}$.

\begin{lemma}
The following inequality holds
\begin{equation}{\label{claim1:intro}}
\Delta_{\varphi,\v}\left(e^{F/2} |dF|_\varphi^{2}\right)\geq - C_1\bigg( |\nabla^0 \omega_\varphi|^2_{g_0\otimes g_\varphi} + \Lambda_0(\omega_\varphi)^{3n-3}|dF|_\varphi^2 + 1\bigg) + C_2 |\nabla^{\varphi,+}dF|_\varphi^{2},
\end{equation}
where $C_1$ is a positive constant depending on $\omega_0,n,\|F\|_{C^0}$, an upper bound for the Ricci curvature, the bounds for $\v,\w$ and their derivatives, while $C_2>0$ depends only on $\|F\|_{C^0}$.
   \end{lemma}
  
   We recall that $$ |\nabla^0\omega_\varphi|^2_{g_0\otimes g_\varphi}= \sum_{k=1}^{n}|\nabla^0_{e_k} \omega_\varphi|^2_{g_\varphi}+\sum_{k=1}^{n}|\nabla^0_{Je_k} \omega_\varphi|^2_{g_\varphi},$$
where $(e_k,Je_k)_{k=1}^{n}$ in a orthonormal frame with respect to $\omega_0$.

\begin{proof}
 We compute

\begin{eqnarray}\label{id lap 1}
\nonumber e^{-F/2}\Delta_{\varphi,\v}\left(e^{F/2} |dF|_\varphi^{2}\right)
&=&\Delta_{\varphi,\v}\left(|dF|_\varphi^{2}\right)+2g_\varphi( d(F/2),d|dF|_\varphi^{2})+|dF|_\varphi^{2}e^{-F/2}\Delta_{\varphi,\v}\left(e^{F/2}\right)\\
&=&\Delta_{\varphi}\left(|dF|_\varphi^{2}\right)+g_\varphi(d\log\v(\mu_\varphi),d|dF|_\varphi^2)+g_\varphi( dF,d|dF|_\varphi^{2}) \\
\nonumber &&+\frac{1}{2}|dF|_\varphi^{2}\Delta_{\varphi,\v} F+\frac{1}{4}|dF|_\varphi^{4},
\end{eqnarray}

\noindent where $|dF|_\varphi^{2}:= g_\varphi(dF,dF)$. In the above we used \eqref{eq:standard} for the first identity and Lemma \ref{lemma weighted Laplacian} to pass to the second line. We now expand each term. \\
Using the fact that the metric $g_\varphi$ is parallel w.r.t. $\nabla_\varphi$ and the expression for the adjoint operator given in \cite[eq. 1.10.13]{Gau}, we arrive at

\[
\begin{split}
\Delta_{\varphi}\left(|dF|_\varphi^{2}\right)=& -d^*_\varphi d\left( |dF|_\varphi^{2}\right)\\
=&-2d^*_\varphi \left( g_\varphi( \nabla^{\varphi}dF,dF)\right)\\
=&2\sum_{k=1}^{2n} \nabla^\varphi_{e_k} \left( g_\varphi( \nabla_{e_k}^{\varphi}dF,dF)\right) \\
=& -2 g_\varphi( d^*_{\varphi}\left(\nabla^\varphi dF\right),dF) + 2|\nabla^\varphi dF|_\varphi^{2}  \\
=&2 g_\varphi(d\Delta_{\varphi}F,dF)+2\ric_{\varphi}(\nabla^{\varphi}F,\nabla^{\varphi}F)+2|\nabla^{\varphi}dF|_\varphi^{2},
\end{split}
\]

\noindent where in the fifth line we used the Bochner formula $-d^*_{\varphi}(\nabla^{\varphi}dF)= \Delta_{\varphi}(dF)+\ric_{\varphi} (\nabla^{\varphi}F, \cdot )$ (cf \cite[eq. 1.22.1]{Gau}). Here $\ric_{\varphi}(\cdot, \cdot)$ denotes the Ricci symmetric $2$-tensor of $g_\varphi$. \\

\noindent For the second term in \eqref{id lap 1}, by Lemma \ref{lemma_id grad} we have

\[
\begin{split}
g_\varphi( d\log \v(\mu_\varphi),d|dF|_\varphi^{2})=&2g_\varphi\big( d\log \v(\mu_\varphi),g_\varphi(\nabla^\varphi dF,dF)\big) \\
=& 2g_\varphi( dF, g_\varphi( d\log(\v(\mu_\varphi), \nabla^\varphi dF)) \\
=&2g_\varphi\big(dF,d \big(g_\varphi (d\log \v(\mu_\varphi),dF)\big)\big) - 2g_\varphi\big( dF, g_\varphi( \nabla^\varphi d \log \v(\mu_\varphi),dF)\big) \\
=&2g_\varphi\big(dF, d\big( g_\varphi (d\log \v(\mu_\varphi),dF)\big)\big)-2\nabla^{\varphi}d\log \v(\mu_\varphi)(\nabla^{\varphi}F,\nabla^{\varphi}F).
\end{split}
\]

\noindent By the fact that $d|dF|_\varphi^{2}= d(g_\varphi(dF, dF))=2g_\varphi(\nabla^\varphi dF, dF)$ and again by Lemma \ref{lemma_id grad}, the next term in \eqref{id lap 1} is 
\[
g_\varphi(dF,d|dF|_\varphi^{2})=2(\nabla^{\varphi}dF)(\nabla^{\varphi}F,\nabla^{\varphi}F).
\]
Moreover, by definition of $\Delta_{\varphi, \v}$ we have
$$g_\varphi(d\Delta_{\varphi}F,dF)+ g_\varphi\big(d\big( g_\varphi (d\log \v(\mu_\varphi),dF)\big), dF\big) =  g_\varphi(d\Delta_{\varphi, \v}F,dF).$$
Substituting back and decomposing $\nabla^\varphi dF$ as a sum of $J$-invarant part $\nabla^{\varphi, +}dF$ and $J$-anti-invariant part $\nabla^{\varphi,-}dF$ we obtain
\[
\begin{split}
e^{-F/2}\Delta_{\varphi,\v}\left(e^{F/2} |dF|_\varphi^{2}\right)=&2 g_\varphi( d\Delta_{\varphi, \v} F,dF)+2\left(\ric_{\varphi}-\nabla^{\varphi}d\log \v(\mu_\varphi)\right)(\nabla^{\varphi}F,\nabla^{\varphi}F)\\
& +2|\nabla^{\varphi,+}dF|_\varphi^{2}+2|\nabla^{\varphi,-}dF|_\varphi^{2}\\
&+2\nabla^{\varphi,+}dF (\nabla^{\varphi}F,\nabla^{\varphi}F)+2\nabla^{\varphi,-}dF (\nabla^{\varphi}F,\nabla^{\varphi}F)\\
& +\left(\frac{1}{2}\Delta_{\varphi,\v} F+ \frac{1}{4}|dF|_\varphi^{2}\right)|dF|_\varphi^{2}.
\end{split}
\]
Notice that by \eqref{eq.grad 2} we have

\begin{equation*}
\begin{split}
    0\leq &2\big|\nabla^{\varphi,-}dF+ \frac{1}{2} ( dF\otimes dF)^-\big|_\varphi^{2}\\
    =&2|\nabla^{\varphi,-}dF|_\varphi^{2} + \frac{1}{8}(|dF|_\varphi^{4} + |d^cF|_\varphi^{4}) \\ &+\big(\nabla^{\varphi,-}dF(\nabla^{\varphi}F,\nabla^{\varphi}F)- \nabla^{\varphi,-}dF(J\nabla^{\varphi}F,J\nabla^{\varphi}F)\big) \\
    =&2|\nabla^{\varphi,-}dF|_\varphi^{2} + \frac{1}{4}|dF|_\varphi^{4} + 2\nabla^{\varphi,-}dF(\nabla^{\varphi}F,\nabla^{\varphi}F),
\end{split}
\end{equation*}

\noindent where $(\alpha\otimes \beta)^-(X,Y):=\frac{1}{2}(\alpha(X)\beta(Y)-\alpha(JX)\beta(JY))$ and  we use that $J\nabla^{\varphi}F$ is the dual  of $d^cF$ with respect to $g_{\varphi}$.
We then get 
\begin{equation}\label{ineq:c2:0}
\begin{split}
e^{-F/2}\Delta_{\varphi,\v}\Big(e^{F/2} |dF|_\varphi^{2}\Big) &\geq  2 g_\varphi( d\Delta_{\varphi,\v} F,dF) +2|\nabla^{\varphi,+}dF|_\varphi^{2}\\
&+2\left(\ric_{\varphi}+\nabla^{\varphi,+}dF-\nabla^{\varphi}d\log \v(\mu_\varphi)\right)(\nabla^{\varphi}F,\nabla^{\varphi}F)\\
&+\frac{1}{2}\Delta_{\varphi,\v} F |dF|_\varphi^{2}.
\end{split}    
\end{equation}
Letting $\Ric(\omega_\varphi)$ be the corresponding Ricci form of $\omega_\varphi$, from \eqref{weighted system} we get 
\begin{equation*}
\Ric(\omega_\varphi) -\frac{1}{2}dd^c\log \v(\mu_\varphi)=\Ric(\omega_0) -\frac{1}{2}dd^cF.
\end{equation*}
Composing with the complex structure $J$ and applying Lemma \ref{ddc-grad} we deduce that

\begin{equation*}
    \ric_{0}=\ric_{\varphi}+\nabla^{\varphi,+}dF -\nabla^{\varphi,+}d\log\v(\mu_\varphi) 
\end{equation*}
Substituting back in \eqref{ineq:c2:0} we obtain
\begin{equation}\label{ineq:c2:1}
\begin{split}
e^{-F/2}\Delta_{\varphi,\v}\left(e^{F/2} |dF|_\varphi^{2}\right)\geq & 2 g_\varphi( d\Delta_{\varphi,\v} F,dF)+2\ric_{0}(\nabla^{\varphi}F,\nabla^{\varphi}F)+ \frac{1}{2}\Delta_{\varphi,\v} F |dF|_\varphi^{2} \\+& 2|\nabla^{\varphi,+}dF|_\varphi^{2} -2\nabla^{\varphi,-}d\log \v(\mu_\varphi)(\nabla^{\varphi}F,\nabla^{\varphi}F)
\end{split}
\end{equation}
We give a bound for the last term of \eqref{ineq:c2:1}. A direct computation shows that

\begin{equation*}
\begin{split}
    \nabla^{\varphi,-}d\log(\v(\mu_\varphi))= &\sum_{b=1}^r\nabla^{\varphi,-} \left( \frac{\v_{,b}(\mu_\varphi)}{\v(\mu_\varphi)}d\mu^b_\varphi  \right) \\
    =&\sum_{b=1}^r\left(  d\left(\frac{\v_{,b}(\mu_\varphi)}{\v(\mu_\varphi)}\right)\otimes d\mu^b_\varphi+  \frac{\v_{,b}(\mu_\varphi)}{\v(\mu_\varphi)}\nabla^{\varphi}d\mu^b_\varphi\right)^-\\
    =& \sum_{a,b=1}^r\left(\frac{\v_{,b}(\mu_\varphi)}{\v(\mu_\varphi)}\right)_{,a}\left(  d\mu^a_\varphi\otimes d\mu^b_\varphi\right)^- + \sum_{b=1}^r \frac{\v_{,b}(\mu_\varphi)}{\v(\mu_\varphi)}\nabla^{\varphi,-}d\mu^b_\varphi\\
    =& \sum_{a,b=1}^r\left(\frac{\v_{,b}(\mu_\varphi)}{\v(\mu_\varphi)}\right)_{,a}\left(  d\mu^a_\varphi\otimes d\mu^b_\varphi\right)^-
\end{split}    
\end{equation*}
To pass to the last line, we use \cite[Lemma 1.23.2]{Gau} which insures that $\nabla^{\varphi,-} d\mu^b_\varphi=0$ since $\mu^b_\varphi$ is Killing potential. On the other hand we have
\[
\begin{split}
    \left|\left(d\mu^a_\varphi\otimes d\mu^b_\varphi\right)^-(\nabla^\varphi F,\nabla^\varphi F)\right|=&\frac{1}{2}|\left(d^cF(\xi_a)d^cF(\xi_b)-dF(\xi_a)dF(\xi_b)\right)|\\
   =&\frac{1}{2}|\left(g_0(J\nabla^0F,\xi_a)g_0(J\nabla^0F,\xi_b)-g_0(\nabla^0F,\xi_a) g_0(\nabla^0F,\xi_b)\right)|\\
   \leq& C|dF|_0^2 \\
   \leq &C_1 |dF|_\varphi^2 \Lambda_0(\omega_\varphi).
\end{split}
\]
In the above we used that $d\mu^a_\varphi(\nabla^\varphi F)=-g_\varphi(J\xi_a, \nabla^\varphi F)= -dF(J\xi_a)=d^c F(\xi_a)$. It follows that
\begin{equation*}
\begin{split}
 |\nabla^{\varphi,-}d\log\v(\mu_\varphi)(\nabla^{\varphi}F,\nabla^{\varphi}F)|\leq&  C_2 |dF|_\varphi^2 \Lambda_0(\omega_\varphi).
\end{split}
\end{equation*}
Notice that
\begin{equation}{\label{ineq:c2:2}}
\begin{split}
|\ric_0 (\nabla^{\varphi}F,\nabla^{\varphi}F)|=&|g_\varphi(\ric_0,dF\otimes dF)|\\
\leq& |\Ric(\omega_0)  |_\varphi |dF|^{2}_\varphi\\
\leq& A_0 |dF|^{2}_\varphi\, \Lambda_\varphi (\omega_0) \\
\leq& A_1 |dF|^{2}_\varphi\, \Lambda_0(\omega_\varphi)^{3n-3}.
\end{split}
\end{equation}
Let us stress that in the second line we apply Cauchy-Schwarz inequality for $2$-tensors and we end up the norm of the $(1,1)$-form $\Ric(\omega_0)$ since 

\begin{equation*}
    |\ric_0|_\varphi =|\ric_0(\cdot , J\cdot)|_\varphi=|\Ric(\omega_0)|_\varphi,
\end{equation*}

\noindent where we use that $J$ is preserving the metric $g_\varphi$.


\noindent In the third inequality we used that $\Ric(\omega_0) \leq A_0\omega_0$ and the well known inequality (see e.g. \cite[(1.12.5)]{Gau})
\begin{equation*}
    |\omega_0|^2_\varphi = \Lambda_\varphi(\omega_0)^2 - \frac{\omega_0^2\wedge \omega_\varphi^{[n-2]}}{\omega_0^{[n]}} \leq \Lambda_\varphi(\omega_0)^2.
\end{equation*}
The last inequality follows from \eqref{ineq:trace:0phi} and \eqref{bound:trace}. Also,

\begin{equation}{\label{ineq:c2:3}}
\begin{split}
\Delta_{\varphi,\v}F|dF|^2 = &\left(-\frac{\w(\mu_\varphi)}{\v(\mu_\varphi)}+2\Lambda_{\varphi,\v}(\Ric(\omega_0))\right)|dF|_\varphi^{2} \\
 \geq & -C\left(\left\| \frac{\w}{\v}\right\|_{C^0}+\Lambda_{\varphi}(\omega_0) +\tilde{C}\right)|dF|_\varphi^{2} \\
 \geq & - C'\left(\left\| \frac{\w}{\v}\right\|_{C^0}+\Lambda_{0}(\omega_\varphi)^{3n-3} + \tilde{C}'\right)|dF|_\varphi^{2}
\end{split}
\end{equation}

\noindent where $C, C'$ depend only on $\omega,\Ric(\omega_0), \v$, and $\tilde{C}$ is the constant which appears in \eqref{bound-trace-weight}. The last line, follows again from \eqref{bound:trace} and \eqref{ineq:trace:0phi}.\\
Using the second equation of \eqref{weighted system} and \eqref{weighted trace} we obtain

\begin{equation}{\label{metric:lap:weighted}}
 \begin{split}
g_\varphi( d\Delta_{\varphi, \v}F,dF)  =&-g_\varphi\left(  d\left(\frac{\w ( \mu_\varphi)}{\v(\mu_\varphi)}\right), dF\right) + 2g_\varphi\big(d \Lambda_{\varphi,\v} ( \Ric(\omega_0 )),dF\big)\\
=& -g_\varphi\left(  d\left(\frac{\w ( \mu_\varphi)}{\v(\mu_\varphi)} +  \frac{1}{2\v(\mu_\varphi)} \sum_{b=1}^r\v_{,b}(\mu_\varphi) \Delta_0 \mu_0^b\right), dF\right) +   \\
&+ 2g_\varphi\big(d \Lambda_{\varphi} ( \Ric(\omega_0 )),dF\big).
\end{split}
\end{equation}

\noindent By \cite[(4.8)]{CC21a}

\begin{equation}{\label{ineq:c2:4}}
     g_\varphi\big(d \Lambda_{\varphi} ( \Ric(\omega_0 )),dF\big) \leq C  \bigg( |\nabla^0\omega_\varphi|^2_{g_0\otimes g_\varphi}+ \left(\Lambda_0(\omega_\varphi)^{3n-3}+1\right)|dF|_\varphi^2 + 1\bigg).
\end{equation}
We now bound the first term of the last line of \eqref{metric:lap:weighted}. \\
Let $f_\varphi:=\frac{\w ( \mu_\varphi)}{\v(\mu_\varphi)} +  \frac{1}{2\v(\mu_\varphi)} \sum_{a=1}^r\v_{,a}(\mu_\varphi) \Delta_0 \mu_0^a$. The differential of $f_\varphi$ is a linear combination of terms of the form $\tilde{f}^a_\varphi d\mu^a_\varphi$, where $\tilde{f}^a_\varphi$ is bounded independently of $\varphi$. Thus, in order to bound $g_\varphi(df_\varphi, dF)$, it is sufficient to bound $g(d\mu^a_\varphi,dF)$ for any $a=1, \cdots, r$. Applying Cauchy-Schwarz and Young inequalities once again we deduce
\begin{equation}{\label{ineq:metric:df}}
    \begin{split}
        g_\varphi(d\mu^a_\varphi, dF)& \leq |d\mu^a_\varphi|_\varphi |dF|_\varphi 
              \\
&=|\xi_a|_\varphi |dF|_\varphi 
              \\              
& \leq \frac{1}{2}|\xi_a|^2_\varphi |dF|^2_\varphi + \frac{1}{2}     
\\
& \leq C \Lambda_0(\omega_\varphi) |dF|^2_\varphi + \frac{1}{2}  \\
& \leq C_1 \Lambda_0(\omega_\varphi)^{3n-3} |dF|^2_\varphi + \frac{1}{2}  
    \end{split}
\end{equation}
where $C$, $C_1$ are constant independent of $\varphi$. Let us observe that in the second line we have $ |d\mu^a_\varphi|_\varphi =|\xi_a|_\varphi$ since the metric is invariant by duality, while in the last line we use \eqref{bound:trace}.

Combining \eqref{ineq:c2:4} and \eqref{ineq:metric:df}, we deduce an upper bound for \eqref{metric:lap:weighted}:

\begin{equation}{\label{ineq:c2:5}}
    g_\varphi( d\Delta_{\varphi, \v}F,dF)  \leq  C  \bigg(|\nabla^0\omega_\varphi|^2_{g_0\otimes g_\varphi} + \left( \Lambda_0(\omega_\varphi)^{3n-3}+1\right)|dF|_\varphi^2 + 1\bigg),
\end{equation}
for a positive constant $C$ independent of $\varphi$.

\medskip

Finally, using  \eqref{ineq:c2:2}, \eqref{ineq:c2:3}, \eqref{ineq:c2:5} we obtain a lower bound for \eqref{ineq:c2:1}:

\begin{eqnarray*}{\label{claim1}}
\Delta_{\varphi, \v} \left(e^{F/2} |dF|_\varphi^{2}\right) &\geq & 
-e^{F/2} C\left(  |\nabla^0\omega_\varphi|^2_{g_0\otimes g_\varphi} +\left(\Lambda_0(\omega_\varphi)^{3n-3}+1\right)|dF|_\varphi^2 + 1\right) \\
&&+  2e^{F/2} |\nabla^{\varphi,+}dF|_\varphi^{2}.
\end{eqnarray*}
This concludes the proof since $\|F\|_{C^0}$ is uniformly under control by Corollary \ref{coro-covid-19} and since $1\leq C \Lambda_0(\omega_\varphi)^{3n-3}$ by \eqref{ineq:trace1}.
\end{proof}

\noindent We are now in position to prove Proposition \ref{p:weightedlap:u}:

\begin{proof}[proof of Proposition \ref{p:weightedlap:u}]
Recall that $u:=e^{\frac{F}{2}} | dF |_\varphi^2 + K \Lambda_0(\omega_\varphi)$. From \eqref{yau:estim2} in Lemma \ref{l:yau} together with \eqref{claim1:intro}, we obtain a lower bound of the weighted Laplacian of $u$:

\begin{equation*}
\begin{split}
    \Delta_{\varphi, \v}\, u \geq  &  - C_1 \bigg( |\nabla^0\omega_\varphi|^2_{g_0\otimes g_\varphi} + \Lambda_0(\omega_\varphi)^{3n-3} |dF |^2_\varphi   + 1  \bigg) +C_2|\nabla^{\varphi,+} dF |_\varphi^2 \\
    & +K|\nabla^0\omega_\varphi|^2_{g_0\otimes g_\varphi} -K\tilde{C_1} \Lambda_0 (\omega_\varphi)^n + K(\Delta_{0} F -\tilde{C_2}).
\end{split}    
\end{equation*}
We now fix $K$ large enough so that one can drop the term $|\nabla^0\omega_\varphi|^2_{g_0\otimes g_\varphi}$.
Observe that, since $n\geq 2$, $C \Lambda_0(\omega_\varphi)^{3n-2} \geq \Lambda_0(\omega_\varphi)^{n} $ thanks to \eqref{bound:trace}. We then deduce 
\begin{equation*}
    \Delta_{\varphi, \v}u \geq  C_2 |\nabla^{\varphi,+} dF |_\varphi^2  - C_3 \Lambda_0(\omega_\varphi)^{3n-3} u +  K(\Delta_{0} F -\tilde{C_2}) - C_1,   
\end{equation*}
where $C_3=\max(C_1\sup_X e^{-F/2}, C\tilde{C}_1 )$.
Observe that $\sup_X e^{-F/2}$ is uniformly under control by Corollary \ref{coro-covid-19}.
Now, in orthonormal frame  $(e_i,Je_i)_{i=1,\cdots,n}$ we find
\begin{equation*}
|\Delta_0 F |=  \frac{1}{2}\left|\sum_{j=1}^{n} {(dd^cF)(e_j,Je_j)} +\sum_{j=1}^{n} {(dd^cF)(Je_j,-e_j)} \right|=\left|\sum_{j=1}^{n} \frac{(dd^cF)(e_j,Je_j)|e_j|_\varphi^2}{|e_j|_\varphi^2}\right|
\end{equation*}
Also by Lemma \ref{ddc-grad}

\begin{equation*}
\begin{split}
    |\nabla^{\varphi,+} dF |_\varphi^2 =& \sum_{j=1}^{n} \frac{(\nabla^{\varphi,+} dF (e_j,e_j))^2}{|e_j|_\varphi^4}+ \frac{(\nabla^{\varphi,+} dF (Je_j,Je_j))^2}{|Je_j|_\varphi^4}\\
    =&2\sum_{j=1}^{n} \frac{(\nabla^{\varphi,+} dF (e_j,e_j))^2}{|e_j|_\varphi^4}\\
    =&\frac{1}{2}\sum_{j=1}^{n} \frac{  (dd^cF(e_j, Je_j))^2}{|e_j|_\varphi^4}.
    \end{split}
\end{equation*}
Hence, by Young's inequality
\begin{eqnarray*}
 K |\Delta_0 F | & \leq & \varepsilon |\nabla^{\varphi,+} dF |_\varphi^2 +  \varepsilon^{-1} K^2 \Lambda_0(\omega_\varphi)^2\\
 &\leq &  \varepsilon |\nabla^{\varphi,+} dF |_\varphi^2 +  C\varepsilon^{-1} K^2 \Lambda_0(\omega_\varphi)^{3n-3} u,
 \end{eqnarray*} 
where in the above we used Lemma \ref{l:trace} (since $n\geq 2$) and the fact that $u$ is uniformly bounded.
Choosing $\varepsilon$ smaller than $C_2$ so that one can drop the term $|\nabla^{\varphi,+} dF |_\varphi^2$, we arrive at
$$\Delta_{\varphi, \v}u \geq  - (C_3+  CK^2\varepsilon^{-1}) \Lambda_0(\omega_\varphi)^{3n-3} u  -K \tilde{C_2} - C_1 \geq - C_4 \Lambda_0(\omega_\varphi)^{3n-3} u,$$
where $C_4$ depends on $\omega_0,K, \varepsilon, n, \v $ and $\|F\|_{C^0}$.
\end{proof}




Another key lemma is the following:

\begin{lemma}
The following $L^1$-estimate holds:

\begin{equation*}
    \|u\|_{L_1} \leq C( \omega_0,\v, \|F\|_{C^0}, \Ent(\varphi)).
\end{equation*}
\end{lemma}

\begin{proof}
Using the second equation in \eqref{weighted system}:
\begin{equation*}
\begin{split}
    \Delta_{\varphi,\v} F^2 &= 2F \Delta_{\varphi,\v} F + 2 |dF|_\varphi^2 \\
    &= 2F \left( -\frac{\w(\mu_\varphi)}{\v(\mu_\varphi)} + 2 \Lambda_{\varphi, \v}(\Ric(\omega_0)) \right) +  2 |dF|_\varphi^2.
\end{split}    
\end{equation*}
Since $\Delta_{\v,\varphi}$ is self-adjoint with respect to $\v(\mu_\varphi)\omega_\varphi^{[n]}$ (Lemma \ref{l:weightedlap:2}), we deduce
\begin{equation*}
\begin{split}
     \int_X |dF|_\varphi^2 \v(\mu_\varphi)\omega_\varphi^{[n]} = &     \int_X F \left( \frac{\w(\mu_\varphi)}{\v(\mu_\varphi)} - 2 \Lambda_{\varphi, \v}(\Ric(\omega_0)) \right)  \v(\mu_\varphi)\omega_\varphi^{[n]} \\
     &\leq    \frac{M}{\eta} \int_X |F| e^F\omega_0^{[n]} + \left( 2A_0L\int_X \Lambda_{ \varphi}(\omega_0) \omega_\varphi^{[n]}+ 2C'L \right) \|F\|_{C^0}.
\end{split}     
\end{equation*}

\noindent In the above we use that $\Lambda_{\varphi, \v}(\Ric(\omega_0))\leq A_0 \Lambda_{\varphi}(\omega_0) +C' $ (thanks to \eqref{weighted trace}) where $C'$ depends on $\v$ and its derivatives, and that $\int_X \Lambda_{ \varphi}(\omega_0) \omega_\varphi^{[n]}=1$.
\end{proof}

We can now conclude:

\begin{proof}[Proof of Theorem \ref{t:c2:estim}]
We assume that $\varphi$ is a $(\v,\w)$-cscK metric. We use the Leibniz rule for the $\v$-weighted Laplacian and we find:

\begin{equation*}
    \begin{split}
    \Delta_{\varphi, \v} u^{2p+1}=& u^{2p} \Delta_{\varphi,\v} u + 2p u^{2p-1}|du|_\varphi^2 \\
    =& u^{2p} \Delta_{\varphi,\v} u + \frac{8p}{(2p+1)^2} \left|d\left(u^{p+\frac{1}{2}}\right)\right|_\varphi^2,
    \end{split}
\end{equation*}

\noindent With the above identity and Proposition \ref{p:weightedlap:u} in hands we can deduce the following bound for $p\geq 1$:


\begin{equation*}
\begin{split}
    \int_X \left|d\left(u^{p+\frac{1}{2}}\right)\right|_\varphi^2 \omega_\varphi^{[n]} \leq& \frac{1}{\eta} \int_X \left|d\left(u^{p+\frac{1}{2}}\right)\right|_\varphi^2 \v(\mu_\varphi)\omega_\varphi^{[n]} \\
    =& \frac{(2p+1)^2}{8p\eta} \int_X \left(-u^{2p} \Delta_{\varphi,\v} u + \Delta_{\varphi, \v} u^{2p+1}\right) \v(\mu_\varphi)\omega_\varphi^{[n]} \\
     =& -\frac{(2p+1)^2}{8p\eta} \int_X u^{2p} \Delta_{\varphi,\v} u \\
    \leq &C \int_X \Lambda_0(\omega_\varphi)^{3n-3} u^{2p+1} \omega_\varphi^{[n]},
\end{split}    
\end{equation*}

\noindent  Observe also that at the third line we also make use of the fact that $\Delta_{\varphi, \v}$ is self-adjoint with respect to $ \v(\mu_\varphi)\omega_\varphi^{[n]}$. \\
Once the latter inequality in hand, the exact same arguments in \cite[Proof of Theorem 5.1]{DD21} allow to conclude. Let us point out that to conclude the proof one needs to make use of the integral bound $ \| \Lambda_0 (\omega_\varphi)\|_{L^p} \leq C$ in Theorem \ref{thm:integral}. This is where the assumption on $\log \v$ is needed.
\end{proof}

\appendix 

\section{Weighted Laplacian and weighted trace}{\label{s:appen}}

\subsection{The weighted Laplacian}{\label{a:weighted:lap}}
For any  smooth function $f$ on $X$, the $\v$-weighted Laplacian is defined as

\begin{equation}\label{def weighted lapl}
\Delta_{\varphi,\v} f := -\frac{1}{\v(\mu_\varphi)}d^*_\varphi(\v(\mu_\varphi) df).
\end{equation} 
It follows straightforward from the above expression that:

\begin{lemma}{\label{l:weightedlap:2}}
The weighted Laplacian $\Delta_{ \varphi, \v}$ is self-adjoint and Leibniz rule holds.
\end{lemma}

\begin{lemma}\label{lemma weighted Laplacian}
   For any smooth function $f$ on $X$, the weighted Laplacian $\Delta_{\varphi, \v}$ can be expressed as
\begin{eqnarray*}
\Delta_{\varphi, \v}f &=& \Delta_\varphi f + \frac{1}{\v(\mu_\varphi)} \sum_{a=1}^r \v_{,a}(\mu_\varphi)  d^cf(\xi_a) \\ 
 &=& \Delta_\varphi f + g_\varphi\left( d\log\v(\mu_\varphi),df \right).
\end{eqnarray*}
In particular $\Delta_{\varphi, \v}$ is elliptic. 
\end{lemma}

\begin{proof}
We apply the well-know identity (see e.g. \cite[(1.10.13)]{Gau})

\begin{equation*}
    d^*_\varphi (h \alpha)=  h d^*_\varphi \alpha - \alpha(\nabla^\varphi h),
\end{equation*}
where $h$ is a smooth function and  $\alpha$ a $1$-form, for $h=\v(\mu_\varphi)$ and $\alpha=df$. We then use the fact that $d^cf(\cdot)=Jd f (\cdot)=-df(J\cdot)$ and that $-J\xi_a$ is the gradient of $\mu_\varphi^a$ with respect to $g_\varphi$ since $d\mu_\varphi^a=-\omega_\varphi(\xi_a, \cdot)=-g_\varphi(J\xi_a, \cdot)$.
\end{proof}

\subsection{The weighted trace}{\label{a:weighted:trace}}

Let $f$ be a $\T$-invariant smooth function $f: X \rightarrow \R$. By Cartan's formula and the $\T$-invariance of $d^cf$ we get that

\begin{equation*}
\begin{split}
    dd^cf(\xi, \cdot)&=  \xi \lrcorner dd^cf(\cdot, \cdot) =\mathcal{L}_{\xi}d^cf - d(d^cf(\xi)) =- d(d^cf(\xi)),
\end{split}    
\end{equation*}
for any $\xi \in \mathfrak{t}$. We fix the normalization of the moment map by defining $\mu^\xi_{dd^c f}:= d^c f(\xi)$ for all $\xi\in\tor$.

We consider a smooth path $(\varphi_t)_t$ in $\mathcal{K}(X,\omega_0)^\T$ with variation $\frac{d}{dt} \varphi_t {\big|_{t=0}}=\Dot{\varphi}$. Taking the variation of the weighted volume form along $\varphi_t$ and then computing it at $t=0$, we obtain (see proof of \cite[Lemma 4]{Lah19})

\begin{equation*}
    \frac{d}{dt}\left(\v(\mu_{\varphi_t})\omega_{\varphi_t}^{[n]}\right)_{|_{t=0}}= \bigg( \Lambda_{\varphi_0}(dd^c \Dot{\varphi}) + \langle d\log\v(\mu_{\varphi_0}), \mu_{dd^c \Dot{\varphi}} \rangle \bigg) \v(\mu_{\varphi_t})\omega_{\varphi_t}^{[n]}.
\end{equation*}

The above identity motivates the following definition: for any $(1,1)$-form $\theta$ with moment map $\mu_{\theta}$, we define the $\v$-\textit{weighted trace} of $\theta$ by

\begin{equation}{\label{def:weighted:trace}}
    \Lambda_{\varphi,\v}(\theta):= \Lambda_{\varphi} (\theta) + \langle d\log(\v(\mu_{\varphi}), \mu_{\theta} \rangle 
\end{equation}

\noindent for any $\varphi \in \mathcal{K}(X,\omega_0)^\T$. Observe that the above definition depends on the choice of the normalization of the momentum map $\mu_\theta$.
Equivalently, in a fixed basis $(\xi_a)_{a=1,r}$ of the Lie algebra $\mathfrak{t}$ of $\mathbb{T}$

\begin{equation}{\label{formula:trace}}
    \Lambda_{\varphi,\v}(\theta)= \Lambda_\varphi(\theta) + \frac{1}{\v(\mu_\varphi)} \sum_{a=1}^r\v_{,a}(\mu_\varphi) \mu^a_\theta.
\end{equation}

In particular, since $\mu_{\Ric(\omega_0)}=-\frac{1}{2}\Delta_0 \mu_0$ with $\mu_0$ being the moment map associated to $\omega_0$ (see \cite[Lemma 5 (i)]{Lah19}, we have

\begin{equation}\label{weighted trace}
\Lambda_{\varphi,\v} ( \Ric(\omega_0 ))= \Lambda_\varphi  \Ric(\omega_0 ) - \frac{1}{2\v(\mu_\varphi)} \sum_{j=1}^r\v_{,j}(\mu_\varphi) \Delta_0 \mu_0^j.
\end{equation}

\begin{lemma}{\label{l:trace:ddc}}
For any $\T$-invariant smooth function $f : X \rightarrow \R$, the following identity holds true
\begin{equation*}
    \Delta_{\varphi,\v} f = \Lambda_{\varphi,\v}( dd^c f).
\end{equation*}
\end{lemma}

\begin{proof}
  It is a direct consequence of \eqref{formula:trace} and Lemma \ref{lemma weighted Laplacian}.
\end{proof}

\bibliographystyle{plain}
\bibliography{biblio.bib}
\end{document}